\documentclass[a4paper,11pt]{article}
\usepackage{amsmath, amsfonts, amscd, amssymb, amsthm, enumerate, xypic}

\def\bfB{\mathbf{B}}
\def\bfC{\mathbf{C}}

\newcommand{\Hom}{\operatorname{Hom}}
\newcommand{\Mat}{\operatorname{M}}
\newcommand{\Mats}{\operatorname{S}}

\newcommand{\id}{\operatorname{id}}

\newcommand{\Ker}{\operatorname{Ker}}

\newcommand{\Vect}{\operatorname{span}}
\newcommand{\im}{\operatorname{Im}}

\newcommand{\tr}{\operatorname{tr}}

\newcommand{\rk}{\operatorname{rk}}
\newcommand{\codim}{\operatorname{codim}}
\renewcommand{\setminus}{\smallsetminus}
\renewcommand{\vec}{\overrightarrow}
\newcommand{\modu}{\operatorname{mod}}

\def\F{\mathbb{F}}
\def\K{\mathbb{K}}

\renewcommand{\L}{\mathbb{L}}

\def\calA{\mathcal{A}}
\def\calB{\mathcal{B}}

\def\calL{\mathcal{L}}
\def\calM{\mathcal{M}}

\def\calR{\mathcal{R}}
\def\calS{\mathcal{S}}
\def\calT{\mathcal{T}}

\def\calV{\mathcal{V}}
\def\calW{\mathcal{W}}

\def\lcro{\mathopen{[\![}}
\def\rcro{\mathclose{]\!]}}

\theoremstyle{definition}
\newtheorem{Def}{Definition}[section]
\newtheorem{Not}[Def]{Notation}

\theoremstyle{plain}
\newtheorem{theo}{Theorem}[section]
\newtheorem{prop}[theo]{Proposition}
\newtheorem{cor}[theo]{Corollary}
\newtheorem{lemma}[theo]{Lemma}
\newtheorem{claim}{Claim}

\theoremstyle{plain}

\theoremstyle{remark}
\newtheorem{Rems}{Remarks}
\newtheorem{Rem}[Rems]{Remark}

\title{Quasi-range-compatible affine maps on large operator spaces}
\author{Cl\'ement de Seguins Pazzis\footnote{Universit\'e de Versailles Saint-Quentin-en-Yvelines, Laboratoire de Math\'ematiques
de Versailles, 45 avenue des Etats-Unis, 78035 Versailles cedex, France}
\footnote{e-mail address: dsp.prof@gmail.com}}

\begin{document}


\thispagestyle{plain}

\maketitle

\begin{abstract}
Let $U$ and $V$ be finite-dimensional vector spaces over an arbitrary field $\K$, and
$\calS$ be a subset of the space $\calL(U,V)$ of all linear maps from $U$ to $V$.
A map $F : \calS \rightarrow V$ is called range-compatible when it satisfies
$F(s) \in \im s$ for all $s \in \calS$; it is called quasi-range-compatible when
the condition is only assumed to apply to the operators whose range does not include a fixed $1$-dimensional linear subspace of $V$.
Among the range-compatible maps are the so-called local maps
$s \mapsto s(x)$ for fixed $x \in U$.

Recently \cite{dSPRC1,dSPRC2}, the range-compatible group homomorphisms on $\calS$ were classified when
$\calS$ is a linear subspace of small codimension in $\calL(U,V)$.
In this work, we consider several variations of that problem: we investigate
range-compatible affine maps on affine subspaces of linear operators;
when $\calS$ is a linear subspace, we give the optimal bound on its codimension for all quasi-range-compatible
homomorphisms on $\calS$ to be local.
Finally, we give the optimal upper bound on the codimension of an affine subspace $\calS$ of $\calL(U,V)$
for all quasi-range-compatible affine maps on it to be local.
\end{abstract}



\section{Introduction}

Let $U$ and $V$ be vector spaces over an arbitrary field $\K$.
Let $\calS$ be a subset of $\calL(U,V)$, the space of all linear maps from $U$ to $V$.
A map $F : \calS \rightarrow V$ is called \textbf{range-compatible} whenever
$$\forall s \in \calS, \; F(s) \in \im s.$$
It is called \textbf{local} when there exists a vector $x \in U$ such that
$$\forall s \in \calS, \; F(s)=s(x).$$
Obviously, every local map is range-compatible, but the converse does not hold in general.

In the rest of the text, we systematically assume that $U$ and $V$ are finite-dimensional.

Given respective bases $\bfB$ and $\bfC$ of $U$ and $V$,
to every map $F : \calS \rightarrow V$ is attached a map $G : \calM \rightarrow \K^n$,
in which $\calM$ denotes the set of all matrices representing the operators of $\calS$ in the above bases,
and $G(\Mat_{\bfB,\bfC}(s))=\Mat_{\bfC}(F(s))$ for all $s \in \calS$
(here, $\Mat_{\bfB,\bfC}(s)$ stands for the matrix of $s$ in the bases $\bfB$ and $\bfC$, and
$\Mat_{\bfC}(F(s))$ stands for the matrix of the vector $F(s)$ in the basis $\bfC$).
We shall say that $G$ represents $F$ in the bases $\bfB$ and $\bfC$.
Then, $F$ is range-compatible if and only if $\forall M \in \calM, \; G(M) \in \im M$.

\vskip 3mm
Several authors independently discovered that every range-compatible linear map on $\calL(U,V)$ is local
(this result is implicit in \cite{Dieudonne}, for instance), but the emergence of the concept and its systematic study are very recent.
Besides the simplicity of the formulation of the problem, the main motivation for studying range-compatible maps
is the following:
\begin{itemize}
\item Range-compatible linear maps are connected, through duality, to the notion of algebraic reflexivity for operator spaces (see Section 1.1 of \cite{dSPRC2}).
\item Theorems on range-compatible linear maps are involved in recent advances in the theory of large spaces of matrices with rank bounded above \cite{dSPclass}; they are also involved in advances in the study of invertibility preserving linear maps \cite{dSPlargelinpres}.
\item Finally, theorems on range-compatible semi-linear maps are expected to yield advances in the theory of
full-rank preserving linear maps.
\end{itemize}

Let us first make a quick summary of the state of the art on range-compatible maps.
Here is the basic result:

\begin{theo}\label{totalspace}
Every range-compatible linear map on $\calL(U,V)$ is local.
\end{theo}

This result was extended as follows in \cite{dSPclass}:

\begin{theo}\label{maintheolin1}
Let $\calS$ be a linear subspace of $\calL(U,V)$ such that $\codim \calS \leq \dim V-2$.
Then, every range-compatible linear map on $\calS$ is local.
\end{theo}

This result was later extended to (group) homomorphisms in \cite{dSPRC1}:

\begin{theo}\label{maintheogroup}
Let $\calS$ be a linear subspace of $\calL(U,V)$ such that $\codim \calS \leq \dim V-2$.
Then, every range-compatible homomorphism on $\calS$ is local.
\end{theo}

The upper-bound $\dim V-2$ turned out to be optimal for homomorphisms in general, but not for linear maps.
For linear maps, the following theorem, which was the main result in \cite{dSPRC1}, gives the optimal upper-bound:

\begin{theo}\label{maintheolin2}
Let $\calS$ be a linear subspace of $\calL(U,V)$, with $\dim U \geq 2$.
Assume that $\codim \calS \leq 2\dim V-3$, and that $\codim \calS \leq 2 \dim V-4$ if $\# \K=2$.
Then, every range-compatible linear map on $\calS$ is local.
\end{theo}

In this article, we shall focus on variations of the above theorems
that are motivated by research on large spaces of matrices with rank bounded above.
The theorems proven here will allow us to extend a famous theorem of
Atkinson and Lloyd \cite{AtkLloyd} in two directions: first, we will be able to extend it
to affine subspaces of matrices over arbitrary fields, and then we will double the range
of dimensions for which the classification of bounded rank spaces is known over all fields.

Here, we consider three main problems. Firstly, we want to consider range-compatible affine maps on affine subspaces of operators.
Secondly, we want to consider a variation of the problem in which
range-compatibility is replaced with the following weaker property:

\begin{Def}
Let $\calS$ be a subset of $\calL(U,V)$, and $D$ be a $1$-dimensional linear subspace of $V$.
A map $F : \calS \rightarrow V$ is called \textbf{quasi-range-compatible with respect to $D$} when
$$\forall s \in \calS, \; D \not\subset \im s \Rightarrow F(s) \in \im s.$$
The map $F$ is called \textbf{quasi-range-compatible} when it is quasi-range-compatible with respect to some $1$-dimensional linear subspace of $V$.
\end{Def}

Finally, we shall tackle a mix of the above two problems, by considering quasi-range-compatible affine maps on large affine subspaces
of $\calL(U,V)$.

Those problems are studied in Sections \ref{RCaffine}, \ref{quasiRC} and \ref{quasiRCaffine}, respectively.
Before we tackle the first one, some general considerations that apply to all of them are needed:
we devote the next section to them.

Before we start, let us state a short remark that will be often used in this work:
let $\calS$ be a \emph{linear} subspace of $\calL(U,V)$ and $F : \calS \rightarrow V$ be an \emph{affine} range-compatible map.
Then, $F$ must be linear. Indeed, $F$ must map the zero operator to the sole vector of its range, that is $0$.
The same holds if $F$ is only assumed to be quasi-range-compatible.

\section{Quotient and splitting space techniques}

\subsection{Quotient space techniques}

We recall the following notation and result from \cite{dSPRC1}:

\begin{Not}
Let $\calS$ be a subset of $\calL(U,V)$ and $V_0$ be a linear subspace of $V$.
Denote by $\pi : V \rightarrow V/V_0$ the canonical projection.
Then, we set
$$\calS \modu V_0:=\bigl\{\pi \circ s \mid s \in \calS\bigr\},$$
which is a subset of $\calL(U,V/V_0)$.
\end{Not}

In particular, $\calS \modu V_0$ is a linear (respectively, affine) subspace of $\calL(U,V/V_0)$ whenever $\calS$ is a linear
(respectively, affine) subspace of $\calL(U,V)$.

\begin{lemma}[Lemma 2.6 in \cite{dSPRC1}]\label{quotientgeneral}
Let $\calS$ be a linear subspace of $\calL(U,V)$ and $F : \calS \rightarrow V$ be a range-compatible homomorphism.
Then, there is a unique range-compatible homomorphism $F \modu V_0 : \calS \modu V_0 \rightarrow V/V_0$
such that
$$\forall s \in \calS, \; (F \modu V_0)(\pi \circ s)=\pi(F(s)),$$
i.e.\ such that the following diagram is commutative
$$\xymatrix{
\calS \ar[rr]^F \ar[d]_{s \mapsto \pi \circ s} & & V \ar[d]^\pi \\
\calS \modu V_0 \ar[rr]_{F \modu V_0} & & V/V_0.
}$$
\end{lemma}

When $V_0=\K y$ for some non-zero vector $y$, we simply write $\calS \modu y$ instead of $\calS \modu \K y$, and $F \modu y$ instead of
$F \modu \K y$.

The key to the existence of $F \modu V_0$ is the fact that $\pi(F(s))=0$ for every $s \in \calS$ whose range is included in
$V_0$.

In this work, we shall need various generalizations of the above lemma.

\begin{lemma}\label{affinequotient}
Let $\calS$ be an affine subspace of $\calL(U,V)$ and $V_0$ be a linear subspace of $V$.
Let $F : \calS \rightarrow V$ be a range-compatible affine map. Denote by $S$ the translation vector space of $\calS$ and by
$\overrightarrow{F}$ the linear part of $F$.
Assume that $\overrightarrow{F}(s) \in V_0$ for every operator $s \in S$ such that $\im s \subset V_0$.
Then, there is a unique range-compatible affine map $F \modu V_0 : \calS \modu V_0 \rightarrow V/V_0$
such that
$$\forall s \in \calS, \; (F \modu V_0)(\pi \circ s)=\pi(F(s)),$$
where $\pi$ denotes the canonical projection of $V$ onto $V/V_0$.
\end{lemma}

\begin{proof}
The uniqueness part is obvious. For the existence, let $s_1$ and $s_2$ be operators in $\calS$ such that
$\pi \circ s_1=\pi \circ s_2$. Then, $s_1-s_2$ belongs to $S$ and its range is included in $V_0$, whence
$\overrightarrow{F}(s_1-s_2) \in V_0$. Thus, $\pi(F(s_1))=\pi(F(s_2))$.
Therefore, we obtain a map $F \modu V_0 : \calS \modu V_0 \rightarrow V/V_0$ such that
$$\forall s \in \calS, \; (F \modu V_0)(\pi \circ s)=\pi(F(s)).$$
One easily checks that $F \modu V_0$ is affine and range-compatible.
\end{proof}

Lemma \ref{quotientgeneral} can also be generalized to quasi-range-compatible homomorphisms:

\begin{lemma}\label{quasiquotient}
Let $\calS$ be a linear subspace of $\calL(U,V)$, and
$F : \calS \rightarrow V$ be a homomorphism that is quasi-range-compatible with respect to some $1$-dimensional linear subspace $D$ of $V$.
Let $V_0$ be a linear subspace of $V$ that does not include $D$.
Then, there is a unique quasi-range-compatible homomorphism $F \modu V_0 : \calS \modu V_0 \rightarrow V/V_0$
such that
$$\forall s \in \calS, \; (F \modu V_0)(\pi \circ s)=\pi(F(s)),$$
where $\pi$ denotes the canonical projection of $V$ onto $V/V_0$.
\end{lemma}

\begin{proof}
For all $s \in \calS$ such that $\im s \subset V_0$, we know that $F(s)\in V_0$ since $D \not\subset \im s$,
and hence $\pi(F(s))=0$. It follows that there exists a unique homomorphism $F \modu V_0 : \calS \modu V_0 \rightarrow V/V_0$
such that
$$\forall s \in \calS, \; (F \modu V_0)(\pi \circ s)=\pi(F(s)).$$
Let us check that $F \modu V_0$ is quasi-range-compatible with respect to the $1$-dimensional subspace $\pi(D)$.
Let $s \in \calS$ be such that $\pi(D)\not\subset \im (\pi \circ s)$.
Then, $D \not\subset \im s$, and hence $F(s) \in \im s$, which in turn yields $(F \modu V_0)(\pi \circ s) \in \pi(\im s)=\im(\pi \circ s)$.
Therefore, $F \modu V_0$ is quasi-range-compatible.
\end{proof}

Finally, we give a version of Lemma \ref{quasiquotient} for quasi-range-compatible affine maps:

\begin{lemma}\label{quasiaffinequotient}
Let $\calS$ be an affine subspace of $\calL(U,V)$ and
$F : \calS \rightarrow V$ be an affine map that is quasi-range-compatible with respect to some $1$-dimensional linear subspace $D$ of $V$.
Denote by $S$ the translation vector space of $\calS$ and by
$\overrightarrow{F}$ the linear part of $F$.
Let $V_0$ be a linear subspace of $V$ that does not include $D$.
Assume that $\overrightarrow{F}(s) \in V_0$ for every operator $s \in S$ such that $\im s \subset V_0$.
Then, there is a unique quasi-range-compatible affine map $F \modu V_0 : \calS \modu V_0 \rightarrow V/V_0$
such that
$$\forall s \in \calS, \; (F \modu V_0)(\pi \circ s)=\pi(F(s)),$$
where $\pi$ denotes the canonical projection of $V$ onto $V/V_0$.
\end{lemma}

The proof is an easy adaptation of the ones of Lemmas \ref{affinequotient} and \ref{quasiquotient}, and we shall therefore omit it.

\subsection{Orthogonality and quotient spaces}

The bilinear form
$$\begin{cases}
\calL(U,V) \times \calL(V,U) & \longrightarrow \K \\
(s,t) & \longmapsto \tr(t \circ s)
\end{cases}$$
is non-degenerate on both sides. In the rest of the article, we consider the orthogonality relation
between operators in $\calL(U,V)$ and operators in $\calL(V,U)$ that is inherited from this bilinear form.
In particular, given a linear subspace $\calS$ of $\calL(U,V)$, one defines
$$\calS^\bot:=\bigl\{t \in \calL(V,U) : \; \forall s \in \calS, \quad \tr(t \circ s)=0\bigr\}$$
and one obtains that $\calS^\bot$ is a linear subspace of $\calL(V,U)$ such that
$$\dim \calS+\dim \calS^\bot=\dim \calL(U,V).$$
Moreover, for every linear subspace $\calS$ of $\calL(U,V)$, combining the above results on dimensions with
the identity $\forall (s,t)\in \calL(U,V) \times \calL(V,U), \; \tr(t \circ s)=\tr(s \circ t)$ yields
$$(\calS^\bot)^\bot=\calS.$$

Let $\calS$ be an affine subspace of $\calL(U,V)$, with translation vector space denoted by $S$, and let $y$ be a non-zero vector of $V$.
Denote by $\pi : V \rightarrow V/\K y$ the canonical projection.
One sees that
$$\dim S^\bot y+\dim \{s \in S : \; \pi \circ s=0\}=\dim U.$$
Thus, by the rank theorem,
$$\codim (\calS \modu y)=\codim \calS - \dim S^\bot y.$$

\subsection{Splitting techniques}

\begin{Not}
Let $m,n,p,q$ be non-negative integers. Given respective subsets $\calA$ and $\calB$ of $\Mat_{m,n}(\K)$ and $\Mat_{p,q}(\K)$, one
sets
$$\calA \vee \calB:=\biggl\{\begin{bmatrix}
A & C \\
0 & B
\end{bmatrix} \mid A \in \calA, \; B \in \calB, \; C \in \Mat_{m,q}(\K)\biggr\},$$
which is a subset of $\Mat_{m+p,n+q}(\K)$.
\end{Not}

\begin{Not}
Let $n,p,q$ be non-negative integers, and $\calA$ and $\calB$ be respective subsets of $\Mat_{n,p}(\K)$ and $\Mat_{n,q}(\K)$.
We define
$$\calA \coprod \calB :=\Bigl\{\begin{bmatrix}
A & B
\end{bmatrix} \mid (A,B) \in \calA \times \calB\Bigr\},$$
which is a subset of $\Mat_{n,p+q}(\K)$.
\end{Not}

The first three results in the following lemma were established in \cite{dSPRC1} (Lemma 2.5), the fourth one is essentially obvious:

\begin{lemma}[Splitting Lemma]\label{splittinglemma}
Let $n,p,q$ be non-negative integers, and $\calA$ and $\calB$ be linear subspaces, respectively, of $\Mat_{n,p}(\K)$ and $\Mat_{n,q}(\K)$.

Given maps $f : \calA \rightarrow \K^n$ and $g : \calB \rightarrow \K^n$,
set
$$f \coprod g : \begin{bmatrix}
A & B
\end{bmatrix} \in \calA \coprod \calB \longmapsto f(A)+g(B).$$
Then:
\begin{enumerate}[(a)]
\item The linear maps (respectively, the homomorphisms) from $\calA \coprod \calB$ to $\K^n$
are the maps of the form $f \coprod g$, where $f$ and $g$ are linear maps (respectively, homomorphisms) defined on $\calA$ and $\calB$, respectively.
\item Given $f \in \Hom(\calA,\K^n)$ and $g \in \Hom(\calB,\K^n)$, the homomorphism
$f \coprod g$ is range-compatible if and only if $f$ and $g$ are range-compatible.
\item Given $f \in \Hom(\calA,\K^n)$ and $g \in \Hom(\calB,\K^n)$, the homomorphism
$f \coprod g$ is local if and only if $f$ and $g$ are local.
\item Given $f \in \Hom(\calA,\K^n)$ and $g \in \Hom(\calB,\K^n)$, if the homomorphism
$f \coprod g$ is quasi-range-compatible then so are $f$ and $g$.
\end{enumerate}
\end{lemma}

\section{Range-compatible affine maps}\label{RCaffine}

\subsection{Main results}\label{affinesection1}

\begin{theo}\label{affinegeneralfield}
Let $U$ and $V$ be finite-dimensional vector spaces, and $\calS$ be an affine subspace of $\calL(U,V)$.
Assume that $\codim \calS \leq \dim V-2$. Then, every range-compatible affine map on $\calS$ is local.
\end{theo}

For fields with more than $2$ elements, this result can be improved as follows:

\begin{theo}\label{affinemorethan2}
Assume that $\K$ has more than $2$ elements. Let $\calS$ be an affine subspace of $\calL(U,V)$
with $\codim \calS \leq \dim V-1$. Then, every range-compatible affine map on $\calS$ is local.
\end{theo}

When $\K=\F_2$, the provision on the codimension of $\calS$ in Theorem \ref{affinegeneralfield} is optimal, as shown by the following example:
with $n \geq 2$, one considers the affine subspace $\calS$ of $\Mat_{n,p}(\K)$ consisting of all the matrices of the form
$$\begin{bmatrix}
? & [?]_{1 \times (p-1)} \\
1 & [?]_{1 \times (p-1)} \\
[0]_{(n-2) \times 1} & [?]_{(n-2) \times (p-1)}
\end{bmatrix}.$$
One checks that the affine map
$$F : (m_{i,j}) \in \calS \mapsto \begin{bmatrix}
0 \\
m_{1,1}+1 \\
[0]_{(n-2) \times 1}
\end{bmatrix}$$
is range-compatible (indeed, $F(M)$ is always a scalar multiple of the first column of $M$).
However, one also checks that $F$ is non-local.
This example allows us to point out the main difficulty in the affine case:
the linear part of an affine range-compatible map might not be range-compatible!

When $\K$ has more than $2$ elements, the provision on the codimension of $\calS$ from Theorem \ref{affinemorethan2}
is optimal, as the following example demonstrates: with $n \geq 2$ and $p \geq 2$, denote by
$\calT$ the affine subspace of $\Mat_{n,p}(\K)$ consisting of all the matrices of the form
$$\begin{bmatrix}
1 & [?]_{1 \times (p-1)} \\
[0]_{(n-1) \times 1} & [?]_{(n-1) \times (p-1)}
\end{bmatrix}.$$
The affine map
$$F : (m_{i,j}) \in \calT \mapsto \begin{bmatrix}
m_{2,2} \\
[0]_{(n-1) \times 1}
\end{bmatrix}$$
is range-compatible since $F(M)$ is always a scalar multiple of the first column of $M$.
However, one checks that $F$ is non-local.

\subsection{A lemma}

The following lemma will help us use quotient space techniques when we deal with range-compatible affine maps.

\begin{lemma}\label{quasiRCrank1lemma}
Let $\calS$ be an affine subspace of $\calL(U,V)$
with $\codim \calS \leq \dim V-1$.
Let $F : \calS \rightarrow V$ be a range-compatible affine map. Denote by $S$ the translation vector space of $\calS$
and by $\overrightarrow{F}$ the linear part of $F$. Then:
\begin{enumerate}[(a)]
\item For every $1$-dimensional linear subspace $D$ of $V$ that is included in the kernel of every operator of $S^\bot$ with rank $1$,
and for every rank $1$ operator $t \in S$ such that $\im t \neq D$, we have $\vec{F}(t)\in \im t$.
\item There exists a $1$-dimensional linear subspace $D$ of $V$ that is included in the kernel of every operator of $S^\bot$ with rank $1$.
\item If $\codim \calS<\dim V-1$ then $\vec{F}(t) \in \im t$ for every rank $1$ operator $t \in S$.
\end{enumerate}
\end{lemma}

\begin{proof}
Let us start with a simple observation. Let $H$ be a linear hyperplane of $V$ which includes the range of some operator
$s \in \calS$. Then, for every $t \in S$ such that $\im t \subset H$, we have $\im(s+t) \subset H$, and hence
$\overrightarrow{F}(t)=F(s+t)-F(s)$ belongs to the linear subspace $H$.
We shall say that $H$ is \textbf{good} if it includes the range of some operator in $\calS$, and bad otherwise.
Thus, given an operator $t \in S$, if we can find good hyperplanes $H_1,\dots,H_p$ of $V$ such that
$\im t=\underset{i=1}{\overset{p}{\bigcap}} H_i$,
then $F(t) \in \im t$.

Next, for a given linear hyperplane $H$ of $V$ to be bad, it is necessary (although not sufficient)
that $H$ be the kernel of some operator in $S^\bot$.
Indeed, if no operator in $S^\bot$ has kernel $H$, then, for every linear map $g  : U \rightarrow V/H$, there
is some $s \in S$ such that the $\pi_H \circ s=g$, where $\pi_H : V \rightarrow V/H$ denotes the canonical projection onto $V/H$;
fixing an arbitrary operator $s_1 \in \calS$, we can find $s \in S$ such that $\pi_H \circ s_1=\pi_H \circ s$, and
hence the range of the operator $s_1-s$ of $\calS$ is included in $H$.

In the linear subspace of $S^\bot$ spanned by its rank $1$ operators, let us pick a basis
$(t_1,\dots,t_p)$ consisting of rank $1$ operators, and let us
set
$$W:=\underset{i=1}{\overset{p}{\bigcap}} \Ker (t_i).$$
Clearly, $W$ is the intersection of all kernels of the operators of $S^\bot$ with rank at most $1$.

Now, we can prove statement (a).
Let $D \subset W$ be a $1$-dimensional linear subspace.
Let $t \in S$ be a rank $1$ operator such that $\im t \neq D$.
For a linear subspace $L$ of $V$, denote by $L^o$ its orthogonal subspace in the dual space $V^\star$ of $V$
(recall that $L^o$ is the set of all linear forms on $V$ that vanish everywhere on $L$).
Then, $(\im t)^o$ and $D^o$ are distinct hyperplanes of $V^\star$,
and hence $(\im t)^o \cap D^o$ is a linear hyperplane of $(\im t)^o$.
Thus, we can find a basis $(\varphi_1,\dots,\varphi_q)$ of $(\im t)^o$ in which no vector belongs to $D^o$: it follows from duality theory that
$\im t=\underset{k=1}{\overset{q}{\bigcap}} \Ker \varphi_k$; thus,
$\im t$ is the intersection of a family of good hyperplanes, and hence $\overrightarrow{F}(t) \in \im t$.
Statement (a) is now proved.

Next, as $\dim S^\bot=\codim \calS <\dim V$, we find that
$$\dim W \geq \dim V-\dim S^\bot>0,$$
which proves statement (b).

Assume finally that $\codim \calS \leq \dim V-2$. Then, $\dim W \geq 2$ and we can therefore pick distinct
$1$-dimensional linear subspaces $D_1$ and $D_2$ of $W$.
Applying point (a) to both $D_1$ and $D_2$ shows that $\vec{F}(t)\in \im t$ for every rank $1$ operator $t$ in $S$
(as then $\im t \neq D_1$ or $\im t \neq D_2$).
\end{proof}

\subsection{Proofs of Theorems \ref{affinegeneralfield} and \ref{affinemorethan2}}

Remember that if $\calS$ is a linear subspace of $\calL(U,V)$, then every range-compatible affine map on $\calS$
is actually linear.

Now, let us prove Theorems \ref{affinegeneralfield} and \ref{affinemorethan2}.
In both proofs, the translation vector space of $\calS$ is denoted by $S$.

\begin{proof}[Proof of Theorem \ref{affinegeneralfield}]
Let $F : \calS \rightarrow V$ be a range-compatible affine map.

We prove the result by induction on $\dim V$.
If $\dim V \leq 1$, the result is vacuous.
Assume now that $\dim V \geq 2$. If $\calS=\calL(U,V)$ then $F$ is linear and
Theorem \ref{totalspace} shows that it is local. Assume now that $\calS \subsetneq \calL(U,V)$.

Let $D$ be a $1$-dimensional linear subspace of $V$, and denote by $\pi : V \rightarrow V/D$ the canonical projection.
Combining point (c) of Lemma \ref{quasiRCrank1lemma} with Lemma \ref{affinequotient}, we obtain a range-compatible affine map
$$F \modu D : \; \calS \modu D \rightarrow V/D$$
such that
$$\forall s \in \calS, \; \pi(F(s))=(F \modu D)(\pi \circ s).$$
Note that if $D=\K y$ and $S^\bot y \neq \{0\}$, then
$$\codim (\calS \modu y) =\codim \calS -\dim(S^\bot y) \leq \dim (V/D)-2.$$
Now, as $S^\bot \neq \{0\}$ and $\dim V \geq 2$ we can choose non-collinear vectors $y_1$ and $y_2$ in $V$
such that $S^\bot y_1 \neq \{0\}$ and $S^\bot y_2 \neq \{0\}$.
By induction $F \modu y_1$ and $F \modu y_2$ are local, which yields vectors $x_1$ and $x_2$ in $U$
such that
$$\forall s \in \calS, \; F(s)-s(x_1) \in \K y_1 \quad \text{and} \quad F(s)-s(x_2)\in \K y_2.$$

Assume that $x_1 \neq x_2$. Then, $s(x_1-x_2) \in \Vect(y_1,y_2)$ for all $s \in \calS$.
Thus, $\calS$ is included in the vector space $\calS'$ of all operators $s \in \calL(U,V)$ such that $s(x_1-x_2) \in \Vect(y_1,y_2)$,
and it is obvious that this space has codimension $\dim V-2$ in $\calL(U,V)$. Then, $\calS=\calS'$
and hence $\calS$ contains $0$. It follows that $F$ is linear, and we deduce from
Theorem \ref{maintheolin1} that it is local.

Assume finally that $x_1=x_2$. Then, as $\K y_1 \cap \K y_2=\{0\}$ we deduce that $F(s)=s(x_1)$ for all $s \in \calS$,
whence $F$ is local.
\end{proof}

\begin{Rem}
It is easy to check from the above proof that Theorem \ref{affinegeneralfield} can be generalized to range-compatible \emph{semi-affine} maps in the following sense: A semi-affine map from $\calS$ to $V$ is defined as a map for which there is a (group) homomorphism $\overrightarrow{F} : S \rightarrow V$
such that $F(t)-F(s)= \overrightarrow{F}(t-s)$ for all $s$ and $t$ in $\calS$.
\end{Rem}

\begin{proof}[Proof of Theorem \ref{affinemorethan2}]
Let $F : \calS \rightarrow V$ be a range-compatible affine map, the linear part of which we denote by $\overrightarrow{F}$.
We use an induction on $\dim V$ to show that $F$ is local.
If $\codim \calS \leq \dim V-2$ then Theorem \ref{affinegeneralfield} readily shows that $F$ is local.
Therefore, in the rest of the proof we assume that $\codim \calS = \dim V-1$.
If $\dim V=1$, then $\calS=\calL(U,V)$; then, $F$ is linear and the statement is known by Theorem \ref{totalspace}.

\vskip 3mm
Assume now that $\dim V=2$. Then $\codim \calS=1$, and hence $S^\bot$ has dimension $1$.
Either $S^\bot$ is spanned by a rank $1$ operator, or it is spanned by a rank $2$ operator.
We tackle each case separately.

\noindent
\textbf{Case 1:} $S^\bot$ contains no rank $1$ operator. \\
Then, point (a) of Lemma \ref{quasiRCrank1lemma} applies to any $1$-dimensional linear subspace of $V$,
and by choosing two of these subspaces we deduce that
$\overrightarrow{F}$ is range-compatible on the rank $1$ operators of $S$.
As $\overrightarrow{F}$ is obviously range-compatible on the operators of rank $0$ or $2$,
we deduce that it is range-compatible.
Since $F$ is an affine map and $F(0)=0$ if $0\in \calS$, there is a unique linear map $\widetilde{F}$ on $\calT:=\Vect(\calS)$
whose restriction to $\calS$ is $F$: noting that the restriction of $\widetilde{F}$ to $S$ is $\overrightarrow{F}$,
one deduces that $\widetilde{F}$ is range-compatible (indeed, for all $t \in \calT$, either $t \in S$ or $\lambda\,t \in \calS$
for some non-zero scalar $\lambda$). Then, by Theorem \ref{maintheolin2}, the map $\widetilde{F}$
is local (note that $2 \dim V-3=\dim V-1$ here), and hence $F$ is local.

\vskip 3mm
\noindent
\textbf{Case 2:} $S^\bot$ contains a rank $1$ operator. \\
Then, $S$ is represented by the matrix space $\K \vee  \Mat_{1,p-1}(\K)$.
If $\calS$ is a linear subspace then again $F$ is local by Theorem \ref{maintheolin2}.
Assume now that $\calS$ is not a linear subspace. Then, in well-chosen bases $(e_1,\dots,e_p)$ and $(f_1,f_2)$ of $U$ and $V$, the operator space
$\calS$ is represented by the set $\calM$ of all matrices of the form
$$\begin{bmatrix}
? & [?]_{1 \times (p-1)} \\
1 & [?]_{1 \times (p-1).}
\end{bmatrix}$$
Any rank $1$ operator in $S^\bot$ has kernel $\K f_1$:
by point (a) of Lemma \ref{quasiRCrank1lemma},
we find that $\overrightarrow{F}$ maps every rank $1$ operator to a vector of its range provided that this range differs from $\K f_1$.
Now, denote by $G : \calM \rightarrow \K^2$ the map that is attached to $F$ in the bases  $(e_1,\dots,e_p)$ and $(f_1,f_2)$,
and by $\overrightarrow{G}$ its linear part. Fix $i \in \lcro 2,n\rcro$.
Then,
$$G_i : X\in \K^2 \mapsto \overrightarrow{G}\Bigl(\begin{bmatrix}
[0]_{2 \times (i-1)} & X & [0]_{2 \times (n-i)}
\end{bmatrix}\Bigr)$$
is linear and maps every vector of $\K^2 \setminus (\K \times \{0\})$ to a scalar multiple of itself.
As $\K$ has more than $2$ elements, we obtain that $G_i$ has at least three pairwise non-collinear eigenvectors,
and hence $G_i=\lambda_i\,\id_{\K^2}$ for some $\lambda_i \in \K$.
Then, denoting by $x$ the vector of $U$ with coordinates $0,\lambda_2,\dots,\lambda_p$ is the basis $(e_1,\dots,e_p)$,
we can replace $F$ with $s \mapsto F(s)-s(x)$.
In this reduced situation, $\overrightarrow{F}$ vanishes at every operator of $S$ that vanishes at $e_1$.
It follows that $F(s)$ is an affine function of $s(e_1)$ only. Returning to $G$, this reads
$$G : (m_{i,j}) \mapsto \begin{bmatrix}
a\,m_{1,1}+b \\
c\,m_{1,1}+d
\end{bmatrix}$$
for some fixed $(a,b,c,d)\in \K^4$. Now, for an arbitrary $\lambda \in \K$, by applying $G$ to $\begin{bmatrix}
\lambda & [0]_{1 \times (p-1)} \\
1 & [0]_{1 \times (p-1)}
\end{bmatrix}$ we deduce that
$$\begin{vmatrix}
\lambda & a \lambda+b \\
1 & c \lambda+d
\end{vmatrix}=0.$$
Thus,
$$\forall \lambda \in \K, \; c\lambda^2+(d-a)\lambda-b=0.$$
As $\K$ has more than $2$ elements, this yields $c=b=0$ and $d=a$, which shows that $F : s \mapsto s(ae_1)$, and hence $F$ is local.

\vskip 3mm
In the remainder of the proof, we assume that $\dim V \geq 3$.

Lemma \ref{quasiRCrank1lemma} yields a $1$-dimensional linear subspace $D_0$ of $V$ such that
$\overrightarrow{F}(t) \in \im t$ for all $t \in S$ such that $\rk t=1$ and $\im t \neq D_0$.

If there existed a basis $(y_1,\dots,y_n)$ of $V$ such that $S^\bot y_i=\{0\}$ for all $i\in \lcro 1,n\rcro$,
then $S^\bot=\{0\}$, contradicting our assumption that $\codim \calS=\dim V-1$.
Thus, we can find a linear hyperplane $H$ of $V$ such that
$\dim S^\bot y>0$ for all $y \in V \setminus H$.
Then, as $\dim V \geq 2$ we know from Lemma 2.5 of \cite{dSPfeweigenvalues} that $V \setminus (H \cup D_0)$
is not included in a linear hyperplane of $V$. This yields linearly independent vectors $y_1,y_2,y_3$
of $V$ such that, for all $i \in \lcro 1,3\rcro$, $y_i \not\in D_0$ and $\dim S^\bot y_i>0$.
Thus, for all $i \in \{1,2,3\}$, Proposition \ref{affinequotient} shows that
the map $F$ induces a range-compatible affine map on $\calS \modu y_i$;
on the other hand $\codim(\calS \modu y_i) \leq \codim \calS -1 \leq \dim(V \modu y_i)-1$,
and by induction we deduce that $F \modu y_i$ is local.

This yields vectors $x_1,x_2,x_3$ such that $F(s)=s(x_i)$ mod $\K y_i$ for all $i \in \{1,2,3\}$.
If there are distinct indices $i$ and $j$ such that $x_i=x_j$, then $F : s \mapsto s(x_i)$.
Assume now that $x_1,x_2,x_3$ are pairwise distinct. Without loss of generality
we can assume that $x_3=0$ (replacing $F$ with $s \mapsto F(s)-s(x_3)$), in which case
$x_1 \neq 0$, $x_2 \neq 0$ and $x_1-x_2 \neq 0$.
Then, for all $s \in \calS$, we have
$$s(x_1) \in \Vect(y_1,y_3), \; s(x_2) \in \Vect(y_2,y_3) \quad \text{and} \quad s(x_1-x_2) \in \Vect(y_1,y_2).$$
If $x_1$ and $x_2$ are collinear, we deduce that, for all $s \in \calS$, the vector
$s(x_1)$ belongs to $\Vect(y_1,y_3) \cap \Vect(y_2,y_3) \cap \Vect(y_1,y_2)=\{0\}$; then, we learn
that $\codim \calS \geq \dim V$, contradicting our assumptions. \\
It follows that $x_1$ and $x_2$ are not collinear, and hence the linear subspace $\calT$ consisting
of all the operators $t \in \calL(U,V)$ such that $t(x_1) \in \Vect(y_1,y_3)$ and $t(x_2) \in \Vect(y_2,y_3)$
has codimension $2(\dim V-2)$ in $\calL(U,V)$. Therefore, $\codim \calS \geq 2 \dim V-4$.
As $\codim \calS \leq \dim V-1$ and $\dim V \geq 3$, we deduce that $\codim \calS =2 \dim V-4$,
which leads to $\calS=\calT$. In particular, $\calS$ is a linear subspace of $\calL(U,V)$, and hence Theorem \ref{maintheolin2}
yields that $F$ is local. This completes our inductive proof.
\end{proof}

\section{Quasi-range-compatible homomorphisms}\label{quasiRC}

\subsection{Statement of the main results}\label{quasiRCsection1}

In this section, we tackle quasi-range-compatible homomorphisms on large linear subspaces of operators.
The following theorem gives the optimal upper-bound on the codimension of $S$
for all quasi-range-compatible homomorphisms on $S$ to be local.

\begin{theo}\label{quasitheo1}
Let $U$ and $V$ be finite-dimensional vector spaces and $S$ be a linear subspace of $\calL(U,V)$.
Assume that $\codim S \leq \dim V-2$, and that $\codim S \leq \dim V-3$ if $\# \K=2$. \\
Then, every quasi-range-compatible homomorphism on $S$ is local.
\end{theo}

The following examples show that the upper-bound on the codimension of $S$ is optimal
for homomorphisms as well as for linear maps.
Consider first the subspace $S:=\K \vee \Mat_{n-1,p-1}(\K)$ of $\Mat_{n,p}(\K)$, which has codimension $n-1$ in $\Mat_{n,p}(\K)$,
and the linear map
$$F : (m_{i,j}) \in S \longmapsto \begin{bmatrix}
0 \\
m_{1,1} \\
[0]_{(n-2) \times 1}
\end{bmatrix}.$$
The map $F$ is quasi-range-compatible with respect to the $1$-dimensional subspace $D:=\K \times \{0\}$ of $\K^n$:
indeed, if the range of some $M\in S$ does not include $D$, then the first column of $M$ equals zero, and hence $F(M)=0 \in \im M$.
However, it is easy to check that $F$ is non-local.

Assume now that $\# \K =2$ and consider the space $\calT:=\K^2 \vee \Mat_{n-2,p-1}(\K)$, which has codimension $n-2$ in $\Mat_{n,p}(\K)$,
and the linear map
$$F : M \in \calT \longmapsto \begin{bmatrix}
m_{1,1} \\
[0]_{(n-1) \times 1}
\end{bmatrix}.$$
Set $x:=\begin{bmatrix}
1 & 1 & [0]_{1 \times (n-2)}
\end{bmatrix}^T$ and $D:=\K x$. Then, $F$ is quasi-range-compatible with respect to $D$: indeed, if
the range of some $M\in \calT$ does not include $D$, then the first column of $M=(m_{i,j})$ equals either $\begin{bmatrix}
0 & m_{2,1} & [0]_{1 \times (n-2)}
\end{bmatrix}^T$, in which case $F(M)=0$, or $\begin{bmatrix}
m_{1,1} & [0]_{1 \times (n-1)}
\end{bmatrix}^T$, in which case $F(M)$ is the first column of $M$.
However, it is easy to check that $F$ is non-local.

In the following propositions, we generalize these examples:

\begin{prop}\label{quasidegenerate1}
Let $U$ and $V$ be finite-dimensional vector spaces and $S$ be a linear subspace of $\calL(U,V)$.
Assume that $\codim S \leq 2\dim V-3$, and that $\codim S \leq 2\dim V-4$ if $\# \K=2$.
Assume that there is a vector $x \in U$ such that $\dim S x=1$, and let $F :
S \rightarrow V$ be a homomorphism that is both non-local and quasi-range-compatible with respect to some
$1$-dimensional linear subspace $D$ of $V$.
Then:
\begin{enumerate}[(a)]
\item Either $F$ is the sum of a local map with a map of the form $s \mapsto \varphi(s(x))$, where
$\varphi$ is a non-linear endomorphism of $S x$;
\item Or $D=S x$ and $F$ is the sum of a local map with a map of the form
$s \mapsto \varphi(s(x))$, where $\varphi$ is a homomorphism from $D$ to $V$.
\end{enumerate}
\end{prop}

Note that in case (a) $F$ is actually range-compatible (this is connected to case (c) of Theorem 1.6 in \cite{dSPRC1}).

\begin{prop}\label{quasidegenerate2}
Let $U$ and $V$ be finite-dimensional vector spaces, and $S$ be a linear subspace of $\calL(U,V)$,
where $\# \K=2$. Assume that $\codim S \leq 2\dim V-5$
and that there is a vector $x \in U$ such that $\dim S x=2$. Let $F :
S \rightarrow V$ be a homomorphism that is both non-local and quasi-range-compatible with respect to some
$1$-dimensional linear subspace $D$ of $V$.
Then:
\begin{enumerate}[(a)]
\item $D \subset S x$;
\item $F$ is the sum of a local map with a map of the form $s \mapsto \varphi(s(x))$, where
$\varphi$ is a rank $1$ projection of $S x$ whose eigenspaces differ from $D$.
\end{enumerate}
\end{prop}

\begin{Def}
Let $U$ and $V$ be finite-dimensional vector spaces, $S$ be a linear subspace of $\calL(U,V)$, and $i \in \{1,2\}$.
We say that $S$ is \textbf{special of type $i$} when there exists $x \in U$ such that $\dim S x=i$.
\end{Def}

\begin{theo}\label{quasitheo2}
Let $U$ and $V$ be finite-dimensional vector spaces, and $S$ be a linear subspace of $\calL(U,V)$.
\begin{enumerate}[(a)]
\item If $\#\K>3$ and $\codim S \leq 2 \dim V-3$ and $\dim V=2$, then either $S$ is special of type $1$ or
every quasi-range-compatible linear map on $S$ is local.
\item If $\#\K>2$ and $\codim S \leq 2 \dim V-4$, then either $S$ is special of type $1$ or
every quasi-range-compatible homomorphism on $S$ is local.
\item If $\# \K=2$ and $\codim S \leq 2 \dim V-6$, then either $S$ is special of type $1$ or $2$,
or every quasi-range-compatible homomorphism on $S$ is local.
\end{enumerate}
\end{theo}

Let us show that the upper-bounds on $\codim S$ from this last theorem are optimal
and that they cannot be improved by considering linear maps instead of homomorphisms.
First of all, let us consider the $3$-dimensional space $S$ of all $3$ by $2$ matrices of the form
$$\begin{bmatrix}
0 & -a \\
a & 0 \\
b & c
\end{bmatrix} \quad \text{with $(a,b,c)\in \K^3$.}$$
The map
$$F : \begin{bmatrix}
0 & -a \\
a & 0 \\
b & c
\end{bmatrix} \in S \mapsto \begin{bmatrix}
-b \\
-c \\
0
\end{bmatrix}$$
is obviously linear and non-local. We claim however that it is quasi-range-compatible with respect to $D:=\K e_3$,
where $e_3:=\begin{bmatrix}
0 & 0 & 1
\end{bmatrix}^T$. To see this, we fix a non-zero matrix $M=\begin{bmatrix}
0 & -a \\
a & 0 \\
b & c
\end{bmatrix}$ in $S$.
Note that the $3$ by $3$ matrix $\begin{bmatrix}
M & F(M)
\end{bmatrix}$ is alternating, and hence its rank equals at most $2$.
If $\rk M=2$, then we deduce that $F(M) \in \im M$.
Assume now that $\rk M=1$. Then, we must have $a=0$ (judging from the upper $2 \times 2$ submatrix),
and hence $\im M=D$. Thus, we have shown that $F(M) \in \im M$ whenever $D \not\subset \im M$
(the case of the zero matrix being trivial). As claimed, $F$ is quasi-range-compatible.
For all integers $n \geq 3$ and $p \geq 2$, we deduce from the splitting lemma
that there is a non-local quasi-range-compatible linear map on $S \vee \Mat_{n-3,p-2}(\K)$.
Yet $S \vee \Mat_{n-3,p-2}(\K)$ has codimension $2n-3$ in $\Mat_{n,p}(\K)$,
and it is obvious that it is not of special type $1$ (one checks that $\dim S X \geq 2$ for all $X \in \K^2 \setminus \{0\}$,
leading to $\dim (S \vee \Mat_{n-3,p-2}(\K))X \geq 2$ for all $X \in \K^p \setminus \{0\}$).
Therefore, the upper bound from point (b) of Theorem \ref{quasitheo2} is optimal if $\dim V>2$.

If $\# \K=3$, the following example shows that the upper-bound from point (b) in Theorem \ref{quasitheo2} is optimal even if $\dim V=2$.
 The space $\Mats_2(\K)$ of all $2$ by $2$ symmetric matrices with entries in $\K$, seen as a linear subspace of $\calL(\K^2,\K^2)$, does not have special type 1.
Consider the linear map
$$F : \begin{bmatrix}
a & b \\
b & c
\end{bmatrix} \in \Mats_2(\K) \longmapsto \begin{bmatrix}
c-a \\
0
\end{bmatrix}.$$
Obviously, $F$ is non-local.
Let $M=\begin{bmatrix}
a & b \\
b & c
\end{bmatrix}$ be a rank $1$ symmetric matrix such that $\im M \neq \{0\} \times \K$.
As $M$ is singular we have $ac=b^2$. If $b=0$ then $F(M) \in \K \times \{0\}=\im M$.
Otherwise $ac=1$ and hence $a$ is non-zero. Since $\# \K=3$ this yields $c=a^{-1}=a$,
and hence $F(M)=0$. In any case, we see that $F(M) \in \im M$.
We conclude that $F$ is quasi-range-compatible (with respect to $\{0\} \times \K$).
As before, for all $p \geq 2$, we can extend this counter-example to obtain a non-local quasi-range-compatible linear map on
some linear hyperplane of $\Mat_{2,p}(\K)$ that does not have special type 1.

In the case when $\dim V=2$, assume that there is a $2$-dimensional field extension $\L$ of $\K$.
We can naturally embed $\L$ as a linear subspace of $\calL_\K(\L)$, and consider an arbitrary non-zero linear form
$\varphi : \L \rightarrow \K$. It is obvious that $\varphi$ is range-compatible (since every non-zero scalar of $\L$
is identified with an automorphism of the $\K$-vector space $\L$) and non-local.
Conversely, it is possible to show that if $\dim V=2$ and if there exists a linear subspace $S$ with codimension $2$
in $\calL(U,V)$ that does not have special type 1 and on which there exists a non-local quasi-range-compatible linear map,
then there exists a $2$-dimensional field extension of $\K$.

Assume finally that $\# \K=2$.
Consider the space
$$\calT:=\Biggl\{\begin{bmatrix}
a & b \\
b & c \\
e & f
\end{bmatrix}\mid (a,b,c,e,f)\in \K^5\Biggr\} \subset \Mat_{3,2}(\K)$$
and the mapping
$$F : \begin{bmatrix}
a & b \\
b & c \\
e & f
\end{bmatrix} \in \calT \mapsto \begin{bmatrix}
a+b \\
0 \\
0
\end{bmatrix}.$$
Obviously, $F$ is linear but non-local.
We claim however that $F$ is quasi-range-compatible with respect to the $1$-dimensional subspace $D:=\K \begin{bmatrix}
1 & 0 & 1
\end{bmatrix}^T$ of $\K^3$.
Let $M = \begin{bmatrix}
a & b \\
b & c \\
e & f
\end{bmatrix} \in \calT$. Assume that $F(M) \not\in \im M$ and that $D \not\subset \im M$.
Let us write $S(M):=\begin{bmatrix}
a & b \\
b & c
\end{bmatrix}$.
Note that $a+b=1$ as $F(M) \neq 0$.
\begin{itemize}
\item If $S(M)$ has rank $1$, then with $a+b=1$, the only possibility is that $a=1$ and $b=c=0$;
Then, $e=1$ and $f=0$ since otherwise $F(M) \in \im M$; then, $D=\im M$, contradicting our assumptions.
\item If $S(M)$ has rank $2$, then so does $M$; our assumptions then lead to
$$\begin{vmatrix}
a & b & 1\\
b & c & 0 \\
e & f & 0
\end{vmatrix}=1 \quad \text{and} \quad \begin{vmatrix}
a & b & 1\\
b & c & 0 \\
e & f & 1
\end{vmatrix}=1;$$
developing both determinants along the last column leads to $\det S(M)=0$, contradiction the assumption that $\rk S(M)=2$.
\end{itemize}
Thus, $F$ is quasi-range-compatible with respect to $D$.

As above, for all integers $n \geq 3$ and $p \geq 2$, we deduce that there is a non-local quasi-range-compatible linear map on
$\calT \vee \Mat_{n-3,p-2}(\K)$, and we note that $\calT \vee \Mat_{n-3,p-2}(\K)$ has codimension $2n-5$ in $\Mat_{n,p}(\K)$
and that it is neither of special type 1 nor 2 (one checks that $\dim \calT X=3$ for all $X \in \K^2\setminus \{0\}$, leading to $\dim(\calT \vee \Mat_{n-3,p-2}(\K))X \geq 3$
for all $X \in \K^p\setminus \{0\}$).

\vskip 3mm
The rest of this part is organized as follows.
In Section \ref{proofquasitheo1}, we prove Theorem \ref{quasitheo1}.
The result of this theorem is then used, in Section \ref{proofdegenerate},
to derive Propositions \ref{quasidegenerate1} and \ref{quasidegenerate2}.
The remainder of the section consists of the proof of Theorem \ref{quasitheo2}, which is split into four parts:
first, we prove point (a) (Section \ref{proofquasitheo2dim2}); then,
in Section \ref{quasibasiclemmas}, we prove a few basic lemmas that are used in the proof of points (b) and (c),
and finally we perform an inductive proof in Section \ref{proofquasitheo2notF2} for point (b), and in Section
\ref{proofquasitheo2F2} for point (c).

\subsection{Proof of Theorem \ref{quasitheo1}}\label{proofquasitheo1}

In order to prove Theorem \ref{quasitheo1}, we start with the special case when $\dim U=1$.

\begin{lemma}\label{dimU=1lemma}
Let $S$ be a linear subspace of $\calL(U,V)$, where $U$ is a $1$-dimensional vector space and $V$ is an arbitrary vector space.
Assume that $\dim S \geq 2$ if $\# \K>2$, and $\dim S \geq 3$ if $\# \K=2$. \\
Then, every quasi-range-compatible homomorphism on $S$ is local.
\end{lemma}

\begin{proof}
We choose a non-zero vector $x \in U$. Let $F : S \rightarrow V$ be a homomorphism that is quasi-range-compatible with respect to
some $1$-dimensional subspace $D$ of $V$. Consider the isomorphism $\varphi : s \in S \mapsto s(x) \in S x$,
and set $G:=F \circ \varphi^{-1}$, which is a homomorphism from $S x$ to $V$.
For all $y \in S x \setminus D$, we have a uniquely-defined scalar $\lambda_y$ such that $G(y)=\lambda_y\, y$.
Let $(y_1,y_2)$ be a linearly independent pair of vectors of $S x$, and assume
that none of the vectors $y_1,y_2$ and $y_3:=y_1+y_2$ belongs to $D$.
Then,
$$\lambda_{y_3}\, y_1+\lambda_{y_3}\, y_2=\lambda_{y_3}\,y_3=G(y_3)=G(y_1)+G(y_2)=\lambda_{y_1}\, y_1+\lambda_{y_2}\, y_2,$$
which, as $y_1$ and $y_2$ are linearly independent, leads to
$$\lambda_{y_1}=\lambda_{y_3}=\lambda_{y_2.}$$
Now, we fix $y_1 \in S x \setminus D$ and we set $\alpha:=\lambda_{y_1}$.
We wish to prove that $G =\alpha\,\id$. Yet, we have just seen that both morphisms $G$ and $\alpha\,\id$
coincide on $S x \setminus (\K y_1 \cup D \cup (-y_1+D))$. In order to conclude, it remains to prove that
$S x \setminus (\K y_1 \cup D \cup (-y_1+D))$ generates the additive group $S x$.
To do so, we distinguish between two cases.

\vskip 3mm
\noindent \textbf{Case 1:} $\dim S x \geq 3$. \\
Then, $P:=\Vect(y_1)+D$ is a proper subgroup of $S x$, and hence $S x \setminus P$ generates $S x$.
Since $P$ includes $\K y_1 \cup D \cup (-y_1+D)$, the claimed result is proven.

\vskip 3mm
\noindent \textbf{Case 2:} $\dim S x=2$ and $\# \K>2$. \\
Since $\# \K>2$, we can pick $1$-dimensional linear subspaces $D_2$ and $D_3$ of $S x$
such that $\Vect(y_1)$, $D_2$, $D_3$ and $D$ are pairwise distinct.
Then, $D_2 \cup D_3$ generates $S x$. Let $i \in \{2,3\}$. Then, $D_i \cap (\Vect(y_1) \cup D \cup (-y_1+D))$
contains at most one non-zero vector.
Yet, the complement of a proper subgroup of $D_i$ contains at least two non-zero vectors because $\# \K>2$
(if $\# \K=3$ this is deduced from the fact that the sole proper subgroup of $D_i$ is $\{0\}$).
Hence, $D_i$ is generated by $D_i \setminus (\Vect(y_1) \cup D \cup (-y_1+D))$.
We conclude that $S x$ is generated by $S x \setminus (\K y_1 \cup D \cup (-y_1+D))$, as claimed.

\vskip 3mm
Hence, in any case, we obtain that $G=\alpha\,\id_E$.
It follows that $F : s \mapsto s(\alpha x)$, and hence $F$ is local.
\end{proof}

\begin{cor}\label{quasitotalspace}
Let $U$ and $V$ be finite-dimensional vector spaces, with $\dim V \geq 2$ if $\# \K>2$ and $\dim V \geq 3$ otherwise.
Then, every quasi-range-compatible homomorphism on $\calL(U,V)$ is local.
\end{cor}

\begin{proof}
We apply the splitting principle. Set $p:=\dim U$ and $n:=\dim V$.
Then, we need to prove that every quasi-range-compatible homomorphism from $\Mat_{n,p}(\K)$ to $\K^n$ is local.
We split $\Mat_{n,p}(\K)=\K^n \coprod \cdots \coprod \K^n$ (with $p$ copies of $\K^n$).
Lemma \ref{dimU=1lemma} shows that every quasi-range-compatible homomorphism on $\K^n$ is local.
Applying the splitting lemma, we obtain by induction on $p$ that every quasi-range-compatible homomorphism on
$\Mat_{n,p}(\K)$ is local.
\end{proof}

From there, we can prove Theorem \ref{quasitheo1}.

\begin{proof}[Proof of Theorem \ref{quasitheo1}]
We prove the result by induction on $\dim V$.
The case $\dim V \leq 1$ is vacuous.
Assume now that $\dim V\geq 2$.
If $S^\bot=\{0\}$ then $S=\calL(U,V)$ and hence the result follows directly from Corollary \ref{quasitotalspace}.
Assume now that there is a non-zero matrix $N \in S^\bot \setminus \{0\}$.
Set $V_0:=\Ker N$. Let $F : S \rightarrow V$ be a homomorphism that is quasi-range-compatible with respect to some $1$-dimensional linear subspace
$D$ of $V$. By Lemma 2.5 of \cite{dSPfeweigenvalues}, there is a basis $(y_1,\dots,y_n)$ of
$V$ in which no vector belongs to $V_0 \cup D$ (indeed, if $\# \K=2$ our assumptions on $\codim S$ imply that $\dim V \geq 3$).
Then, for all $i \in \lcro 1,n\rcro$, as $N y_i \neq 0$ we see that
$$\codim (S \modu y_i)=\codim S -\dim S^\bot y_i \leq \codim S-1.$$
On the other hand $y_i \not\in D$. Thus, $F \modu y_i$ is well-defined, and by induction it is local.
In particular, this yields vectors $x_1,x_2$ in $U$ such that
$$\forall s \in S, \quad F(s)=s(x_1)\mod \K y_1 \quad \text{and} \quad F(s)=s(x_2)\mod \K y_2.$$
Assume that $x_1=x_2$. Then, $F(s)=s(x_1)$ for all $s \in S$, and hence $F$ is local. \\
Assume finally that $x_1-x_2 \neq 0$. Then, $S$ is included in the linear subspace $\calT$ of $\calL(U,V)$ consisting of all the operators
$t \in \calL(U,V)$ such that $t(x_1-x_2) \in \Vect(y_1,y_2)$. As $\codim \calT=\dim V-2$, we deduce that
$S=\calT$ and $\# \K>2$. In well-chosen bases of $U$ and $V$, the space $\calT$ is represented by
$V_1 \coprod \Mat_{n,p-1}(\K)$, where $V_1$ is the subspace $\K^2 \times \{0\}$ of $\K^n$, and $p=\dim U$.
By Lemma \ref{dimU=1lemma}, every quasi-range-compatible homomorphism on $V_1$ is local, whereas
Corollary \ref{quasitotalspace} yields that every quasi-range-compatible homomorphism on $\Mat_{n,p-1}(\K)$ is local.
Applying the splitting lemma, we conclude that every quasi-range-compatible homomorphism on $S$ is local.
\end{proof}

\subsection{Proofs of Propositions \ref{quasidegenerate1} and \ref{quasidegenerate2}}\label{proofdegenerate}

Here, we will use Theorem \ref{quasitheo1} to study quasi-range-compatible homomorphisms on spaces of special type.

We start with the proof of Proposition \ref{quasidegenerate1}.

\begin{proof}[Proof of Proposition \ref{quasidegenerate1}]
Without loss of generality, we can assume that $U=\K^p$, $V=\K^n$, $x$ is the first vector of the standard basis of
$\K^p$, and $S x=\K \times \{0\}$.

In that canonical situation, $S$ is seen as a linear subspace of $\Mat_{n,p}(\K)$.
For every $M=(m_{i,j}) \in S$, let us write
$$M=\begin{bmatrix}
C_1(M) & J(M)
\end{bmatrix}$$
where $$C_1(M)=\begin{bmatrix}
m_{1,1} \\
[0]_{(n-1) \times 1}
\end{bmatrix} \quad \text{and} \quad
J(M) \in \Mat_{n,p-1}(\K).$$
First of all, let us consider the subspace $\calT$ of all matrices $M \in S$ such that $C_1(M)=0$.
Then, $\dim \calT =\dim S-1$ and $\calT=\{0\} \coprod J(\calT)$.
Theorem \ref{quasitheo1} applies to $J(\calT)$, and hence every range-compatible homomorphism on $J(\calT)$ is local.
This yields a vector $x_1 \in U$ such that $F(t)=t(x_1)$ for all $t \in \calT$.
Replacing $F$ with $s \mapsto F(s)-s(x_1)$, no generality is then lost in assuming that $F(t)=0$ for all $t \in \calT$.
Thus, we find a homomorphism $\varphi : S x \rightarrow V$ such that
$$\forall s \in S, \; F(s)=\varphi(s(x)).$$

Assume first that $\varphi$ maps $S x$ into itself.
If $\varphi$ were linear, then $\varphi : y \mapsto \lambda y$ for some $\lambda \in \K$
and $F$ would be the local map $s \mapsto s(\lambda x)$, contradicting our assumptions.
Thus, $\varphi$ is non-linear, and outcome (a) from Proposition \ref{quasidegenerate1} holds.

\vskip 3mm
In the rest of the proof, we assume that $\varphi(S x)\not\subset S x$, and we prove that outcome (b) holds.
It only remains to prove that $D=S x$.
Assume on the contrary that $D \neq S x$. Then, $S$ does not contain the elementary matrix $E_{1,1}$
(with zero entries everywhere except at the $(1,1)$-spot where the entry equals $1$): indeed, the contrary would yield
that, for every $a \in \K$, the vector $F(a E_{1,1})=\varphi(a\,x)$ belongs to $S x$, resulting in
$\im \varphi \subset S x$. Therefore, $\dim J(S)=\dim S$, and hence
Theorem \ref{quasitheo1} applies to $J(S) \modu S x$.
Choosing a matrix space $\calT'$ which is represented by $J(S)\modu S x$, we see
that $\{0\} \coprod \calT'$ represents $S \modu S x$, whence
$F \modu S x$ is local.
This yields a vector $x' \in U$ such that
$$\forall s \in S, \quad F(s)=s(x') \mod S x.$$
Let us choose $s_1$ in $S$ such that $s_1(x) \neq 0$, and set $P:= \K s_1(x')+ S x$, which is a linear subspace with dimension at most $2$.
Then, for all $s \in S$, there exists $\lambda \in \K$ such that $s(x)=\lambda\, s_1(x)$,
leading to $\varphi(s(x)) = F(\lambda s_1)=\lambda\,s_1(x')\mod S x$.
Hence, $$S x +\im \varphi \subset P.$$
Thus, $S$ is included in the space $\calW$ of all operators $w \in \calL(U,V)$ such that
$w(x) \in S x$ and $w(x')\in P$. Note that $x$ is linearly independent on $x'$ as the contrary would yield $\im \varphi \subset S x$.
It follows that $\calW$ has codimension at least $2\dim V-3$ in $\calL(U,V)$,
whence $\calW=S$. Therefore, $S$ contains $E_{1,1}$, contradicting an earlier result.
This final contradiction leads to $D=S x$, which completes the proof.
\end{proof}

\begin{proof}[Proof of Proposition \ref{quasidegenerate2}]
Without loss of generality, we can assume that $U=\K^p$, $V=\K^n$, $x$ is the first vector of the standard basis of
$\K^p$, and $S x=\K^2 \times \{0\}$. Set $P:=S x$.
In that canonical situation, $S$ is seen as a linear subspace of $\Mat_{n,p}(\K)$, and we write every matrix $M=(m_{i,j})$ of $S$ as
$$M=\begin{bmatrix}
C_1(M) & J(M)
\end{bmatrix}$$
where
$$C_1(M)=\begin{bmatrix}
m_{1,1} \\
m_{2,1} \\
[0]_{(n-2) \times 1}
\end{bmatrix} \quad \text{and} \quad
J(M) \in \Mat_{n,p-1}(\K).$$
As in the proof of Proposition \ref{quasidegenerate1}, we deduce from Theorem \ref{quasitheo1} that
$F$ restricts to a local map on the space of all operators $s \in S$ such that $s(x)=0$,
which allows us to reduce the situation to the one where there is a homomorphism $\varphi : P \rightarrow V$
such that
$$\forall s \in S, \; F(s)=\varphi(s(x)).$$
Note that $\varphi$ is linear since $\K$ is a prime field.

Next, we demonstrate that $D \subset P$ and that $\varphi$ is a projection of $P$
whose eigenspaces differ from $D$.
To do so, we prove that every vector of $P \setminus D$ is an eigenvector of $\varphi$.

Let $y \in P \setminus D$. Without loss of generality, we can assume that $y$ is the first vector of the standard basis of $\K^n$.
If $S$ contains $E_{1,1}$, then as $D \neq \K y$ we find that $F(E_{1,1}) \in \im E_{1,1}=\K y$, that is $\varphi(y) \in \K y$.

Assume now that $S$ does not contain $E_{1,1}$. Then, $\codim(S \modu y) \leq \codim S-1 \leq 2(\dim V-1)-4$, and
$(S \modu y)x=P/\K y$ has dimension $1$. Proposition \ref{quasidegenerate1}
yields a vector $x' \in U$ together with a
linear mapping $\psi : P \rightarrow V$ such that
$$\forall s \in S, \quad F(s)=\psi(s(x))+s(x') \mod \K y$$
and $\psi$ vanishes at $y$ (if $F$ is local we simply take $\psi=0$).
We claim that $x$ and $x'$ are linearly dependent.
Otherwise, we would find that $S$ is included in the space
$\calT$ of all operators $t \in \calL(U,V)$ such that $t(x) \in P$ and $t(x') \in \K y+\im(\varphi-\psi)$, which has codimension at least $2\dim V-5$ in $\calL(U,V)$; then, $S=\calT$, and we would deduce that $S$ contains $E_{1,1}$, contradicting an earlier assumption.
Thus, $x'=\alpha \,x$ for some scalar $\alpha$. Then, by choosing $s \in S$ such that $s(x)=y$, we find $\varphi(y)\in \K y$.

Therefore, every vector of $P \setminus D$ is an eigenvector of $\varphi$.
Yet, $\varphi$ cannot equal $\beta\,\id_P$ for some $\beta \in \K$, as it would yield that $F$ is the local map $s \mapsto s(\beta\,x)$.
Thus, $\varphi$ has several eigenvalues. As $\K=\{0,1\}$, it follows that $D \subset P$ and that $\varphi$
is a projection whose eigenspaces are different from $D$.
\end{proof}

\subsection{Proof of point (a) in Theorem \ref{quasitheo2}}\label{proofquasitheo2dim2}

Here, we assume that $\dim V=2$, $\# \K>3$ and $\codim S \leq 1$.
If $\codim S=0$ then we know from Corollary \ref{quasitotalspace} that every range-compatible homomorphism on $S$
is local. In the rest of the section, we assume that $\codim S=1$, so that $S^\bot$ has dimension $1$.

If $S^\bot$ contains a rank $1$ operator $t$, we choose a non-zero vector $x$ in its range
and we obtain that $\dim S x \leq 1$. As $\codim S=1$ we deduce that $\dim S x=1$, and hence $S$ has special type 1.

In the rest of the proof, we assume that $S^\bot$ contains a rank $2$ operator $t$.
Then, we aim at proving that every quasi-range-compatible linear map on $S$ is local.
Let $F : S \rightarrow V$ be a linear map that is quasi-range-compatible with respect to some $1$-dimensional linear subspace $D_0$
of $V$. Let us choose a basis $\bfC=(y_1,y_2)$ of $V$ and a basis $\bfB$ of $U$ in which $t$ is represented by
$\begin{bmatrix}
0 & -1 \\
1 & 0 \\
[0]_{(p-2) \times 1} & [0]_{(p-2) \times 1}
\end{bmatrix}$ and $y_2 \in D_0$ (where $p:=\dim U$).
It follows that $S$ is represented in the bases $\bfB$ and $\bfC$ by the space $\Mats_2(\K) \coprod \Mat_{2,p-2}(\K)$.
As every quasi-range-compatible homomorphism on $\Mat_{2,p-2}(\K)$ is local, it only remains to prove the following lemma.

\begin{lemma}\label{symlemma}
Let $\K$ be a field with more than $3$ elements.
Every linear map on $\Mats_2(\K)$ that is quasi-range-compatible with respect to $D_0:=\{0\} \times \K$ is local.
\end{lemma}

\begin{proof}
Let $F : \Mats_2(\K) \rightarrow \K^2$ be a linear map that is quasi-range-compatible with respect to $D_0$.
Subtracting a local map from $F$, we find scalars
$\alpha,\beta,\gamma,\delta$ such that
$$F : \begin{bmatrix}
a & b \\
b & c
\end{bmatrix} \mapsto \begin{bmatrix}
\alpha a+\beta b+\gamma c \\
\delta a
\end{bmatrix}.$$
For all $t \in \K$, the range of the matrix $\begin{bmatrix}
1 & t \\
t & t^2
\end{bmatrix}$ is spanned by $\begin{bmatrix}
1 \\
t
\end{bmatrix}$ and is therefore different from $D_0$. It follows that
$$\forall t \in \K, \; \delta=\alpha t+\beta t^2+\gamma t^3.$$
As $\K$ has more than $3$ elements, this polynomial identity leads to $\delta=\alpha=\beta=\gamma=0$.
Thus, $F=0$ and hence $F$ is local.
\end{proof}

This completes the proof of point (a) in Theorem \ref{quasitheo2}.

\subsection{Common lemmas for the proof of Theorem \ref{quasitheo2}}\label{quasibasiclemmas}

Here, we gather several elementary results that will be used in the proof of points (b) and (c) from Theorem \ref{quasitheo2}
and in Section \ref{quasiRCaffine}. 

\begin{lemma}\label{quasialternative}
Let $S$ be a linear subspace of $\calL(U,V)$, where $U$ and $V$ are finite-dimensional vector spaces.
Set $n:=\dim V$ and $p:=\dim U$. \\
Assume that $\dim S^\bot y \leq 1$ for all $y \in V$, and that $\codim S>1$.
Then, in some bases of $U$ and $V$, the space $S$ is represented by
$\calR \coprod \Mat_{n,p-1}(\K)$ for some linear subspace $\calR$ of $\K^n$.
\end{lemma}

\begin{proof}
Every operator in $\widehat{S^\bot}:=\{t \in S^\bot \mapsto t(y) \mid y \in V\}$
has rank at most $1$. By the classification of vector spaces of operators with rank at most $1$,
either all the non-zero operators in $\widehat{S^\bot}$ share the same range, or all of them share the same kernel.
However, no non-zero vector of $S^\bot$ belongs to the kernel of every operator in $\widehat{S^\bot}$
and hence the second option implies that $\dim S^\bot \leq 1$, contradicting our assumptions.
Thus, we find a non-zero vector $x \in U$ such that $S^\bot y \subset \K x$ for all $y \in V$.
Extending $x$ into a basis $(x,e_2,\dots,e_p)$ of $U$, we deduce that $S$ is represented by a space of matrices of the form
$\calR \coprod \Mat_{n,p-1}(\K)$ for some linear subspace $\calR$ of $\K^n$.
\end{proof}

\begin{lemma}\label{quasiexistquadform}
Let $U$ and $V$ be finite-dimensional vector spaces, and $S$ be a linear subspace of $\calL(U,V)$.
Assume that $\dim S^\bot y > 1$ for some $y \in V$.
Then, there is a non-zero quadratic form $q$ on $V$ such that
$\dim S^\bot z \geq 2$ for every vector $z \in V$ such that $q(z) \neq 0$.
\end{lemma}

\begin{proof}
Consider the operator space $\widehat{S^\bot}:=\bigl\{t\in S^\bot \mapsto t(z) \mid z \in V\bigr\}$.
We can choose bases of $S^\bot$ and $U$ in which the operator $t \in S^\bot \mapsto t(y)$ is represented by
$\begin{bmatrix}
I_r & [0] \\
[0] & [0]
\end{bmatrix}$ for some integer $r \geq 2$. Then, to all $z \in V$ we assign the determinant $q(z)$
of the upper-left $2 \times 2$ submatrix of the matrix representing $t\in S^\bot \mapsto t(z)$ in the said bases.
Clearly, $q$ is quadratic form on $V$ and $\rk(t\in S^\bot \mapsto t(z)) \geq 2$ for all $z \in V$ such that $q(z) \neq 0$, that is
$\dim S^\bot z \geq 2$ for all such $z$.
\end{proof}

\begin{lemma}\label{quasipropquadform}
Assume that $\# \K>2$.
Let $q$ be a quadratic form over a finite-dimensional vector space $E$.
Assume that there is a linear hyperplane $H$ of $E$ and two linear subspaces $P_1$ and $P_2$ of codimension $2$ in $E$
such that $q$ vanishes at every vector of $E \setminus (H \cup P_1 \cup P_2)$. Then, $q=0$.
\end{lemma}

\begin{proof}
Firstly, we prove that $q$ vanishes everywhere on $E \setminus (P_1 \cup P_2)$.
Let $x \in H \setminus (P_1 \cup P_2)$. Then, the linear hyperplanes $H$, $P_1+\K x$ and $P_2+\K x$
do not cover $E$ since $\# \K \geq 3$. It follows that we can find $y \in E$ that belongs to none of them.
Then, $Q:=\Vect(x,y)$ is a $2$-dimensional linear subspace that intersects $P_1$ and $P_2$ trivially
and is not included in $H$. The restriction of $q$ to $Q$
vanishes at every $1$-dimensional linear subspace of $Q$ that is distinct from $Q \cap H$, and as $\# \K>2$
there are at least three such subspaces. It follows that the quadratic form $q$ vanishes everywhere on $Q$, and in particular $q(x)=0$.
Hence, $q$ vanishes everywhere on $E \setminus (P_1 \cup P_2)$.

Next, applying the same line of reasoning with an arbitrary linear hyperplane $H'$ that includes $P_1$, and with $P'_1:=P_2$ and
$P'_2:=P_2$, we obtain that $q$ vanishes everywhere on $E \setminus P_2$.

Finally, let us pick an arbitrary non-zero linear form $\varphi$ on $E$ that vanishes everywhere on $P_2$.
We have just shown that $x \mapsto q(x)\varphi(x)$ vanishes everywhere on $E$. However this function is a homogeneous polynomial
with degree $3$. As $\varphi$ is non-zero we conclude that $q=0$.
\end{proof}

Our final lemma is taken from \cite{dSPRC1} and will be useful for fields with two elements.

\begin{lemma}[Lemma 5.2 of \cite{dSPRC1}]\label{quadformlemmaF2}
Let $q$ be a quadratic form over a finite-dimensional vector space $E$, and $P$
be a linear subspace of $E$ with codimension $2$. Assume that $q$ vanishes everywhere on $E \setminus P$. Then, $q=0$.
\end{lemma}

\subsection{Proof of point (b) of Theorem \ref{quasitheo2}}\label{proofquasitheo2notF2}

This section is devoted to the proof of statement (b) in Theorem \ref{quasitheo2}.
Again, this statement is proved by induction on $\dim V$.
Throughout the section, we assume that $\K$ has more than $2$ elements.

The case $\dim V=1$ is vacuous. If $\dim V=2$ then $S=\calL(U,V)$ and hence the result is known from
Corollary \ref{quasitotalspace}.
In the rest of the section, we assume that $\dim V \geq 3$ and we consider a homomorphism $F : S \rightarrow V$
that is quasi-range-compatible with respect to some $1$-dimensional linear subspace $D$ of $V$.
If $\codim S \leq 1$, then Theorem \ref{quasitheo1} shows that $F$ is local.
In the rest of the proof, we assume that $\codim S>1$, that $F$ is non-local and that $S$ does not have special type 1.
We shall find a contradiction. First, some definitions will help:

\begin{Def}
A vector $y$ of $V$ is called \textbf{$S$-adapted} whenever
$y \not\in D$ and $\codim (S \modu y) \leq 2 \dim V-6$. \\
An $S$-adapted vector $y$ is called \textbf{special} when $F \modu y$ is non-local.
\end{Def}

In particular, a vector $y$ of $V \setminus D$ is $S$-adapted if $\dim S^\bot y>1$.

If, in some bases of $U$ and $V$, there existed a linear subspace $\calR$ of $\K^n$ such that $S$
is represented by $\calR \coprod \Mat_{n,p-1}(\K)$, then:
\begin{itemize}
\item Either $\dim \calR=1$, in which case $S$ would have special type 1;
\item Or $\dim \calR \neq 1$, in which case Lemma \ref{dimU=1lemma}, Corollary \ref{quasitotalspace} and the splitting lemma would
yield that every quasi-range-compatible homomorphism on $S$ is local.
\end{itemize}
Thus, by Lemma \ref{quasialternative} there exists a vector $y \in V$ such that $\dim S^\bot y>1$,
and then Lemma \ref{quasiexistquadform} yields a non-zero quadratic form $q$ on $V$ that vanishes at
every vector $z \in V$ such that $\dim S^\bot z \leq 1$.

Assume first that $\dim V=3$. \\
Assume furthermore there are non-collinear $S$-adapted vectors $y_1$ and $y_2$.
Then, for all $i \in \{1,2\}$ we have $\codim (S \modu y_i)=0$ and hence
$S \modu y_i=\calL(U,V/\K y_i)$.
In particular, we know from Corollary \ref{quasitotalspace} that both maps $F \modu y_1$ and $F \modu y_2$ are local,
yielding vectors $x_1$ and $x_2$ in $U$ such that $F(s)=s(x_i) \mod y_i$ for all $s \in S$ and all $i \in \{1,2\}$.
Replacing $F$ with $s \mapsto F(s)-s(x_1)$, no generality is lost in assuming that $x_1=0$.
Then, we see that $s(x_2) \in \Vect(y_1,y_2)$ for all $s \in S$. If $x_2 \neq 0$, then this would
contradict the fact that $S \modu y_1=\calL(U,V/\K y_1)$. Therefore, $x_2=0$, and we conclude that
$F=0$ since $\K y_1 \cap \K y_2=\{0\}$. In particular, $F$ would be local, which is false.

Thus, we find a $1$-dimensional linear subspace $D'$ of $\K^3$ that contains all the $S$-adapted vectors of $V$.
However, it would follow that $q$ vanishes everywhere on $V \setminus (D \cup D')$, which would contradict Lemma \ref{quasipropquadform}.

\vskip 3mm

In the rest of the proof, we assume that $\dim V>3$. Then, we need an additional result:

\begin{claim}\label{quasi3vectors}
There do not exist linearly independent vectors $y_1,y_2,y_3$, all in $V \setminus D$, such that each map $F \modu y_i$ is local.
\end{claim}

\begin{proof}
Assume on the contrary that there exist such vectors $y_1,y_2,y_3$, yielding vectors $x_1,x_2,x_3$ in $U$ such that
$$\forall i \in \{1,2,3\}, \; \forall s \in S, \quad F(s)=s(x_i) \mod \K y_i.$$
If $x_i=x_j$ for some distinct $i$ and $j$ in $\{1,2,3\}$, we deduce that $F : s \mapsto s(x_i)$ since $\K y_i \cap \K y_j=\{0\}$.
In that case, $F$ would be local. As $F$ is non-local $x_1,x_2,x_3$ are pairwise distinct.

Replacing $F$ with $s \mapsto F(s)-s(x_3)$, we lose no generality in assuming that $x_3=0$, in which case
$F(s) \in \K y_3$ for all $s \in S$, and $x_1$ and $x_2$ are distinct non-zero vectors of $U$.
Thus, for all $s \in S$, we find that
$$s(x_1) \in \Vect(y_1,y_3), \quad s(x_2) \in \Vect(y_2,y_3) \quad \text{and} \quad s(x_1-x_2) \in \Vect(y_1,y_2).$$
If $x_1$ and $x_2$ are collinear, the above results yield $s(x_1)=0$ for all $s \in S$, since
$\Vect(y_1,y_3) \cap \Vect(y_2,y_3) \cap \Vect(y_1,y_2)=\{0\}$; this would imply $F=0$,
contradicting the assumption that $F$ is non-local.

Thus, $x_1$ and $x_2$ are linearly independent. The space $S$ is then included in the space $\calT$
of all linear operators $t \in \calL(U,V)$ such that $t(x_1)\in \Vect(y_1,y_3)$ and $t(x_2)\in \Vect(y_2,y_3)$.
Obviously, $\calT$ has codimension $2\dim V-4$ in $\calL(U,V)$, and it follows that $S=\calT$.
This is absurd as $\calT$ contains an operator $t$ such that $t(x_1)=y_3$ and $t(x_2)=0$, so that $t(x_1-x_2) \not\in \Vect(y_1,y_2)$.
\end{proof}

By induction, for every special  $S$-adapted vector $y$, we know that $S \modu y$ has special type 1.
As $F$ is non-local, Claim \ref{quasi3vectors} and the induction hypothesis yield a $2$-dimensional linear subspace $P$ of
$E$ that contains all the non-special $S$-adapted vectors. As $\dim V>3$, we can choose a linear hyperplane $H$ of $E$
that includes $D+P$. Thus, every $S$-adapted vector $y \in V \setminus H$ is a special $S$-adapted vector.

By Lemma \ref{quasipropquadform}, we can choose an $S$-adapted vector $y_1 \in V \setminus H$.
This yields a vector $x_1 \in U$ such that $(S \modu y_1) x_1$ has dimension $1$.
Then, $\dim \bigl((S \modu y_1) x_1\bigr) \leq \dim S x_1 \leq \dim \bigl((S \modu y_1) x_1\bigr)+1$.
As $S$ does not have special type 1, we deduce that $Q_1:=S x_1$ is a $2$-dimensional linear subspace of $V$.
Lemma \ref{quasipropquadform} yields a vector $y_2 \in V \setminus (Q_1 \cup H)$ such that $q(y_2) \neq 0$,
and hence $y_2$ is a special $S$-adapted vector outside of $Q_1$.
Thus, we have a vector $x_2 \in U$ such that $\dim (S \modu y_2)x_2=1$.
As above, $Q_2:=S x_2$ is a $2$-dimensional linear subspace of $V$.
However, as $y_2 \not\in Q_1$ we obtain $\dim (S \modu y_2)x_1=2$, whence
$x_1$ and $x_2$ are non-collinear. Thus, the subspace $\calT$ consisting of all the operators $t \in \calL(U,V)$
such that $t(x_1) \in Q_1$ and $t(x_2) \in Q_2$  has codimension $2\dim V-4$ in $\calL(U,V)$.
Therefore, $S=\calT$, and it follows that there are $2$-dimensional linear subspaces $\calV_1$ and $\calV_2$ of $\K^n$ such that,
in well-chosen bases of $U$ and $V$, the operator space $S$ is represented by the matrix space $\calV_1 \coprod \calV_2 \coprod \Mat_{n,p-2}(\K)$.
Applying Lemma \ref{dimU=1lemma}, Corollary \ref{quasitotalspace} and the splitting lemma, we conclude that every quasi-range-compatible
homomorphism on $S$ is local, contradicting our assumption that $F$ is non-local.

This completes the inductive proof of statement (b) in Theorem \ref{quasitheo2}.

\subsection{Proof of point (c) of Theorem \ref{quasitheo2}}\label{proofquasitheo2F2}

Here, we complete the proof of Theorem \ref{quasitheo2} by tackling statement (c).
Once again, this is done by induction on the dimension of $V$.
The case when $\dim V \leq 2$ is vacuous. If $\dim V=3$ then we must have $S=\calL(U,V)$ and the statement follows from Corollary
\ref{quasitotalspace}. In the rest of the proof, we assume that $\dim V \geq 4$.

We shall perform a \emph{reductio ad absurdum} by assuming:
\begin{enumerate}[(a)]
\item That $S$ has neither special type $1$ nor special type $2$;
\item That there exists a non-local homomorphism $F : S \rightarrow V$ that is quasi-range-compatible with respect to some
$1$-dimensional linear subspace $D$ of $V$.
\end{enumerate}

By Theorem \ref{quasitheo1} and assumption (b), we must have $\codim S>\dim V-3 \geq 1$.
If, in some bases of $U$ and $V$, there existed a linear subspace $\calR$ of $\K^n$ such that $S$
is represented by $\calR \coprod \Mat_{n,p-1}(\K)$, then:
\begin{itemize}
\item We would have $\dim \calR \not\in \{1,2\}$ since $S$ does not have special type $1$ or $2$;
\item Then, Lemma \ref{dimU=1lemma}, Corollary \ref{quasitotalspace} and the splitting lemma would
show that every quasi-range-compatible homomorphism on $S$ is local, contradicting assumption (b).
\end{itemize}
Thus, Lemma \ref{quasialternative} yields a vector $y \in V$ such that $\dim S^\bot y>1$.
By Lemma \ref{quasiexistquadform}, we obtain a non-zero quadratic form $q$ on $V$
that vanishes at every vector $z \in V$ such that $\dim S^\bot z \leq 1$.

\begin{Def}
A non-zero vector $z$ of $V$ is called \textbf{$S$-adapted} whenever
$z \not\in D$ and
$\dim (S \modu z) \leq 2 \dim V-8$.
\end{Def}

Thus, every vector $z \in V \setminus D$ such that $\dim S^\bot z \geq 2$ is $S$-adapted.
Beware that this definition is different from the one adopted in the preceding section.

\begin{claim}\label{distinctvectorsclaim}
There do not exist distinct vectors $y_1$ and $y_2$ in $V \setminus D$
such that both maps $F \modu y_1$ and $F \modu y_2$ are local.
\end{claim}

\begin{proof}
Assume on the contrary that there are two such vectors $y_1$ and $y_2$, yielding vectors $x_1$ and $x_2$ in $U$ such that
$F(s)=s(x_1) \mod \K y_1$ and $F(s)=s(x_2) \mod \K y_2$ for all $s \in S$.
If $x_1=x_2$ then $F(s)=s(x_1)$ for all $s \in S$, contradicting the assumption that $F$ is non-local.
Thus, $x_1-x_2 \neq 0$ and we find that $s(x_1-x_2) \in \Vect(y_1,y_2)$ for all $s \in S$.
As $S$ has neither special type $1$ nor special type $2$, we deduce that $s(x_1-x_2)=0$ for all $s \in S$.
Thus, in well-chosen bases of $U$ and $V$, the space $S$ is represented by the matrix space $\{0\} \coprod \calT$
for some linear subspace $\calT$ of $\Mat_{n,p-1}(\K)$, and $\codim \calT=\codim S-\dim V \leq \dim V-3$.
Theorem \ref{quasitheo1} yields that every quasi-range-compatible homomorphism on $\calT$ is local, and hence
 every quasi-range-compatible homomorphism on $S$ is local. This contradicts our assumption that $F$ is non-local.
\end{proof}

\begin{claim}\label{stype2claim}
For every $S$-adapted vector $y \in V \setminus D$, either $F \modu y$ is local or
$S \modu y$ has special type $2$.
\end{claim}

\begin{proof}
Assume that there is an $S$-adapted vector $y \in V \setminus D$ such that $F \modu y$ is non-local and
$S \modu y$ does not have special type $2$. Then, by induction $S \modu y$ has special type $1$,
yielding a non-zero vector $x \in U$ such that $(S \modu y)x$ has dimension $1$.
Then, $S x$ would have dimension $1$ or $2$, and hence $S$ would be of special type $1$ or of special type $2$,
which has been ruled out.
\end{proof}

\begin{claim}
There exists an $S$-adapted vector $y \in V \setminus D$ such that $F \modu y$ is non-local.
\end{claim}

\begin{proof}
As we can choose a linear subspace of $V$ with codimension $2$ that includes $D$,
Lemma \ref{quadformlemmaF2} yields an $S$-adapted vector $y$ in $V \setminus D$.
Assume that $F \modu y$ is local. Then, with the same line of reasoning, we find another $S$-adapted vector $y'$ in $V \setminus (D+\K y)$.
Thus, by Claim \ref{distinctvectorsclaim}, the map $F \modu y'$ is non-local.
\end{proof}

\begin{claim}
One has $\dim V \geq 5$.
\end{claim}

\begin{proof}
Indeed, if $\dim V=4$ then we choose an $S$-adapted vector $y \in V \setminus D$ such that $F \modu y$ is non-local.
Then, we would have $S \modu y=\calL(U,V/\K y)$ since $\codim(S \modu y) \leq 0$, and by Corollary \ref{quasitotalspace} every quasi-range-compatible homomorphism on $S \modu y$ would be local, contradicting our assumptions.
\end{proof}

Now, we fix $y \in V \setminus D$ that is $S$-adapted and such that $F \modu y$ is non-local.
Thus, by Claim \ref{stype2claim} the space $S \modu y$ has special type $2$.
We choose a vector $x \in U$ such that $\dim (S \modu y)x=2$.
By Proposition \ref{quasidegenerate2}, we lose on generality in assuming that $F \modu y$ has rank $1$
(as we can subtract any local map from $F$) and that it maps every operator of $S$ into $(S \modu y)x$.

Note that $2 \leq \dim S x \leq 3$. As $S$ does not have special type $2$, we find
$\dim S x=3$ and hence $y \in S x$; by statement (a) in Proposition \ref{quasidegenerate2} applied to $F \modu y$,
we know that $(D+\K y)/\K y$ is included in $(S x)/\K y$, which yields $D \subset S x$ since $y \in S x$.
Set $Q:=S x$ and note that $\im F \subset Q$ and $\rk F \in \{1,2\}$.
As $\dim V \geq 5$, Lemma \ref{quadformlemmaF2} yields  a vector $z \in V \setminus Q$ such that
$q(z) \neq 0$, and in particular $z \in V \setminus D$, $z$ is $S$-adapted and $\rk(F \modu z)=\rk(F)$.

Thus, by Claim \ref{stype2claim} either $F \modu z$ is local, or else $S \modu z$ has special type $2$.

\begin{claim}
There is a vector $x' \in U \setminus \K x$ such that $\dim S x' \leq 3$.
\end{claim}

\begin{proof}
Note that $(S \modu z) x$ equals $(Q +\K z)/\K z$, and hence it has dimension $3$.

Assume first that $S \modu z$ has special type $2$, yielding a vector $x' \in U \setminus \{0\}$
such that $\dim (S \modu z) x'=2$. Then, $x' \neq x$ and the conclusion follows immediately.

Assume now that $F \modu z$ is local, yielding a vector $x' \in U$ such that
$F(s)=s(x') \mod \K z$ for all $s \in S$. However $\rk(F \modu z)=\rk F \in \{1,2\}$
rules out the possibility that $x'=0$, but also the one that $x'=x$ since
$(S \modu z)x$ has dimension $3$. Thus, $x$ and $x'$ are linearly independent, and
we deduce from $\rk(F \modu z) \leq 2$ that the space $S x'$ has dimension at most $3$.
\end{proof}

From there, we can obtain a final contradiction:
set $Q':=S x'$, and
denote by $\calT$ the space of all operators $t \in \calL(U,V)$ such that $t(x) \in Q$ and $t(x') \in Q'$.
Note that $\calT$ has codimension $2 \dim V-\dim Q-\dim Q'$ in $\calL(U,V)$. Thus, $S=\calT$, $\dim Q'=3$
and, in some bases of $U$ and $V$, the space $S$ is represented by $\calR_1 \coprod \calR_2 \coprod \Mat_{n,p-2}(\K)$
for some $3$-dimensional linear subspaces $\calR_1$ and $\calR_2$ of $\K^n$. Then,
Lemma \ref{dimU=1lemma}, Corollary \ref{quasitotalspace} and the splitting lemma yield that every quasi-range-compatible homomorphism on
$S$ is local, contradicting our assumption on $F$.
This final contradiction completes our inductive proof of statement (c) of Theorem \ref{quasitheo2}.

\section{Quasi-range-compatible affine maps}\label{quasiRCaffine}

In this final section, we deal with quasi-range-compatible affine maps.
We shall combine some techniques from the previous sections with Theorems \ref{quasitheo1} and \ref{quasitheo2}
to obtain the optimal upper-bound on the codimension of an affine subspace $\calS$ of $\calL(U,V)$ for all
affine quasi-range-compatible maps on it to be local.

\subsection{The results}

Our first result is the equivalent of Theorem \ref{quasitheo1} for affine subspaces of operators.

\begin{theo}\label{quasiaffinetheo1}
Let $\calS$ be an affine subspace of $\calL(U,V)$ such that
$\codim \calS \leq \dim V-2$, or $\codim \calS \leq \dim V-3$ if $\# \K=2$.
Then, every quasi-range-compatible affine map on $\calS$ is local.
\end{theo}

The examples from Section \ref{quasiRCsection1} show that the upper-bounds from Theorem \ref{quasiaffinetheo1}
are optimal. The next two results deal with the critical codimension when the field has more than $2$ elements.

\begin{theo}\label{quasiaffinetheo2}
Let $\calS$ be an affine subspace of $\calL(U,V)$ such that
$\codim \calS=\dim V-1$. Assume further that either $\dim V \geq 3$ and $\# \K>2$,
or $\dim V \geq 2$ and $\# \K>3$.
Then, either every quasi-range-compatible affine map on $\calS$ is local or
there exists a vector $x \in U$ such that $\dim \calS x=1$.
\end{theo}

Again, the examples from Section \ref{affinesection1} show that the upper-bound $\dim V-1$
is optimal. Moreover, the second example from Section \ref{quasiRCsection1} justifies the exclusion of the special
case when $\# \K=3$ and $\dim V=2$.

Under the assumptions of Theorem \ref{quasiaffinetheo2}, if there exists a vector $x \in U$ such that $\calS x$
is a $1$-dimensional \emph{linear} subspace of $V$, then the assumption $\codim \calS \leq \dim V-1$
yields that $\calS$ is actually a linear subspace of $\calL(U,V)$;
the quasi-range-compatible affine maps on $\calS$ are then linear, and their description is given in Proposition \ref{quasidegenerate1}.

Our final result deals with the special case when $\calS x$ is a $1$-dimensional affine subspace of $V$
but not a linear subspace.

\begin{theo}\label{quasiaffinetheo3}
Let $\calS$ be an affine subspace of $\calL(U,V)$ such that
$\codim \calS=\dim V-1$ and $\dim V \geq 2$, and assume that $\# \K>2$.
Assume that there is a vector $x \in \calS$ such that $\dim \calS x=1$
and that $\calS x$ is not a linear subspace of $V$.
Let $F : \calS \rightarrow V$ be a non-local affine map that is quasi-range-compatible with respect to some $1$-dimensional
linear subspace $D_0$ of $V$.
Then:
\begin{enumerate}[(i)]
\item $\# \K=3$;
\item $D_0 \cap \calS x \neq \emptyset$;
\item There is a vector $x' \in U$ such that $F(s)-s(x') \in \Vect(\calS x)$ for all $s \in \calS$;
\item If $\dim V>2$ then there is an endomorphism $\psi$ of $\Vect(\calS x)$ and a vector $x' \in U$ such that
$F(s)=\psi(s(x))+s(x')$ for all $s \in \calS$.
\end{enumerate}
\end{theo}

Assuming that $\# \K=3$, the case $\dim V=2$ can be fully described as follows:
without loss of generality, we can assume that $\calS$ is the space of all matrices of the form
$\begin{bmatrix}
a & [?]_{1 \times (p-1)} \\
1-a & [?]_{1 \times (p-1)}
\end{bmatrix}$
with $a \in \K$, and that $D_0=\K\begin{bmatrix}
1 \\
1
\end{bmatrix}$.
Then, the affine maps from $\calS$ to $\K^2$ that are quasi-range-compatible with respect to $D_0$
are the sums of the local maps with the maps of the form
$$(m_{i,j}) \mapsto \begin{bmatrix}
\varepsilon\,m_{1,1} \\
\underset{j=2}{\overset{p}{\sum}} a_j m_{2,j}
\end{bmatrix}$$
for some $\varepsilon \in \K$ and some $(a_2,\dots,a_p)\in \K^{p-1}$.
As this result does not seem particularly useful, we shall leave the details of its proof to the reader.

\subsection{Proof of Theorem \ref{quasiaffinetheo1}}

The following lemma is the equivalent of Lemma \ref{quasiRCrank1lemma} in the theory of quasi-range-compatible affine maps.

\begin{lemma}\label{secondlinelemma}
Let $\calS$ be an affine subspace of $\calL(U,V)$ such that either $\codim \calS \leq \dim V-3$, or $\codim \calS \leq \dim V-1$ and $\# \K>2$.
Let $F : \calS \rightarrow V$ be an affine map that is range-compatible with respect to some $1$-dimensional linear subspace $D_0$ of $V$.
Denote by $S$ the translation vector space of $\calS$ and by $\vec{F} : S \rightarrow V$ the linear part of $F$.
Let $D_1$ be a $1$-dimensional linear subspace of $V$ that is included in the kernel of every operator of $S^\bot$.

Then, for every $1$-dimensional subspace $D$ of $V$ that is different from both $D_0$ and $D_1$,
$$\forall s \in S, \quad \im s=D \; \Rightarrow \; \vec{F}(s) \in D.$$
\end{lemma}

\begin{proof}
Set $n:=\dim V$.
As in the proof of Lemma \ref{quasiRCrank1lemma}, we consider the span $\calT$ of the rank $1$ operators in $\calS^\bot$,
and we note that $W:=\underset{t \in \calT}{\bigcap} \Ker t$ is a linear subspace of $V$
with dimension at least $\dim V-\codim \calS \geq 1$.

Let $H$ be a linear hyperplane of $V$ that includes neither $D_0$ nor $W$.
Then, on the one hand we have $F(s) \in H$ for all $s \in \calS$ such that $\im s \subset H$
(since $D_0 \not\subset H$) and on the other hand as $W \not\subset H$ there actually exists $s \in \calS$
such that $\im s \subset H$. Then, with the same line of reasoning as in the proof of
Lemma \ref{quasiRCrank1lemma}, we deduce that $\vec{F}(s) \in H$ for all $s \in S$ such that $\im s \subset H$.

Assume first that $\# \K>2$, and let $D_1$ be an arbitrary $1$-dimensional linear subspace of $W$.
Let $D$ be a $1$-dimensional linear subspace of $V$ that differs from $D_0$ and $D_1$.
Then, we can find linear hyperplanes $H_1,\dots,H_{n-1}$ of $V$, each of which includes neither $D_0$ nor $D_1$,
such that $D=\underset{k=1}{\overset{n-1}{\bigcap}} H_k$.
To see this, denote by $L^o$ the dual orthogonal of an arbitrary subspace $L$ of $V$,
as defined by $\{\varphi \in V^\star : \forall y \in L, \; \varphi(y)=0\}$, where $V^\star$ denotes the space of all linear forms on $V$.
Then, $D^o \cap D_0^o$ and $D^o \cap D_1^o$ are linear hyperplanes of $D^o$
and, as $\# \K>2$, Lemma 2.5 of \cite{dSPfeweigenvalues} yields a basis $(\varphi_1,\dots,\varphi_{n-1})$ of $D^o$ in which no vector belongs to
$D_1^o \cup D_0^o$; we simply take $H_i:=\Ker \varphi_i$ for all $i \in \lcro 1,n-1\rcro$.
It follows that for all $s\in S$ such that $\im s=D$, we have
$\vec{F}(s) \in \underset{k=1}{\overset{n-1}{\bigcap}} H_k=D$.

Assume finally that $\# \K=2$. Let $D$ be a $1$-dimensional linear subspace of $V$ that differs from $D_0$.
Then, $\dim W \geq 3$ and hence $D^o \cap W^o$ has codimension at least $2$ in $D^o$, while $D^o \cap D_0^o$ is a linear hyperplane of $D^o$.
By Lemma 2.5 of \cite{dSPfeweigenvalues} we find a basis of $D^o$ in which no vector belongs to $W^o \cup D_0^o$.
It follows that we can find linear hyperplanes $H_1,\dots,H_{n-1}$ of $V$ such that $D=\underset{k=1}{\overset{n-1}{\bigcap}} H_k$
and $D_0 \not\subset H_k$ and $W \not\subset H_k$ for all $k \in \lcro 1,n-1\rcro$. As above we deduce that
$\vec{F}(s) \in D$ for all $s\in S$ such that $\im s=D$. Thus, $D_1$ qualifies.
\end{proof}

From there, the proof of Theorem \ref{quasiaffinetheo1} is very similar to the one of Theorem \ref{quasitheo1}. Again, the proof is done by induction on $\dim V$. The result is vacuous if $\dim V<2$. If $\calS=\calL(U,V)$, then the result follows from Corollary
\ref{quasitotalspace} (remember that every affine quasi-range-compatible map is linear provided that its domain contains the zero operator).
In particular, this yields the result whenever $\dim V=2$.

Assume now that $\dim V>2$ and that $\calS \neq \calL(U,V)$. In particular, if $\# \K=2$ then $\dim V \geq 4$.
Let $F : \calS \rightarrow V$ be an affine map that is quasi-range-compatible with
respect to some $1$-dimensional subspace $D_0$ of $V$. Fix a $1$-dimensional subspace $D_1$ of $V$
given by Lemma \ref{secondlinelemma}. Denote by $S$ the translation vector space of $\calS$.
As $S^\bot \neq \{0\}$, the space $H:=\underset{t \in S^\bot}{\bigcap} \Ker t$ is a proper linear subspace of $V$.
For all $y \in V \setminus H$, we see that
$$\codim(\calS \modu y)<\codim \calS$$
and hence
$$\codim(\calS \modu y) \leq \dim(V/\K y)-2,$$
and
$$\codim(\calS \modu y) \leq \dim(V/\K y)-3 \quad \text{if $\# \K=2$.}$$
By Lemma 2.5 of \cite{dSPfeweigenvalues}, we can pick non-collinear vectors $y_1$ and $y_2$ in $V \setminus (H \cup (D_0+D_1))$
(if $\# \K=2$ this uses the fact that $\dim V \geq 4$).
By Lemma \ref{quasiaffinequotient}, the affine maps $F \modu y_1$ and $F \modu y_2$ are well-defined and quasi-range-compatible.
By induction we find vectors $x_1$ and $x_2$ in $U$ such that
$$\forall s \in \calS, \quad F(s)=s(x_1) \mod \K y_1 \quad \text{and} \quad F(s)=s(x_2) \mod \K y_2.$$
If $x_1=x_2$, then $F$ is the local map $s \mapsto s(x_1)$.

Assume now that $x_1 \neq x_2$. Then, $\calS$ is included in the linear subspace $\calT$ of all $t \in \calL(U,V)$
such that $t(x_1-x_2) \in \Vect(y_1,y_2)$. As $\calT$ has codimension $\dim V-2$ in $\calL(U,V)$
we conclude that $\calS=\calT$. Thus, $\calS$ is actually a linear subspace of $\calL(U,V)$, and Theorem \ref{quasitheo1}
yields that $F$ is local.

This completes the inductive proof of Theorem \ref{quasiaffinetheo1}.

\subsection{Proof of Theorem \ref{quasiaffinetheo2}}

This time around, the proof is not done by induction. Rather, we use the full force of earlier results.

Denote by $S$ the translation vector space of $\calS$.
Let us assume that there is an affine map $F : \calS \rightarrow V$ that is non-local
but that is quasi-range-compatible with respect to some $1$-dimensional linear subspace $D_0$ of $V$.
We denote by $\vec{F} : S \rightarrow V$ its linear part. We seek to find a vector $x \in U$ such that $\dim \calS x=1$.

If $U=\{0\}$ then $\calS=\{0\}$ and hence $F=0$, which contradicts our assumption that $F$ is non-local.
Thus $\dim U>0$. If $\dim V=1$, then $\calS=\calL(U,V)$ and we immediately find a vector $x \in U$ such that $\dim \calS x=1$.

\vskip 3mm
Assume now that $\dim V=2$ and $\# \K>3$. Then, $\dim S^\bot=1$.
If $S^\bot$ contains a rank $1$ operator, then we choose a non-zero vector $x$ in its range, and we obtain that $\dim \calS x=1$.
Assume now that $S^\bot$ contains no rank $1$ operator:
then, in Lemma \ref{secondlinelemma} we can take $D_1=D_0$, and hence $\vec{F}$ is quasi-range-compatible with respect to $D_0$.
We can therefore extend $F$ into a quasi-range-compatible linear map on $\Vect(\calS)$ (this uses the fact that $F(0)=0$ if $0 \in \calS$).
There are two options: if $\Vect(\calS)=\calL(U,V)$ then Theorem \ref{quasitheo1} shows that the linear extension of $F$ to
$\calS$ is local, which yields that $F$ is local, contradicting our assumptions; otherwise
$\calS$ is a linear subspace of $\calL(U,V)$, and point (a) of Theorem \ref{quasitheo2}
shows that $\calS x$ has dimension $1$ for some $x \in U$ (since $F$ is a non-local quasi-range-compatible linear map on $\calS$).

\vskip 3mm
In the rest of the proof, we assume that $\dim V>2$ and $\# \K>2$.
Lemma \ref{secondlinelemma} yields a $1$-dimensional linear subspace $D_1$ of $V$
such that $\vec{F}(s) \in \im s$ for all $s \in S$ such that $\im s$ has dimension $1$ and differs from both $D_0$ and $D_1$.
We shall say that a vector $y \in V \setminus \{0\}$ is \textbf{$\calS$-adapted} whenever
$\dim S^\bot y \geq 2$. For any such vector $y$, provided that it lay outside of $D_0 \cup D_1$,
the affine map $F \modu y$ is well-defined and quasi-range-compatible, whereas $\codim (\calS \modu y) \leq \dim (V \modu \K y)-2$,
and hence Theorem \ref{quasiaffinetheo1} yields that $F \modu y$ is local.
Assume now that we can find three linearly independent $\calS$-adapted vectors $y_1,y_2,y_3$ in $V \setminus (D_0 \cup D_1)$.
This yields vectors $x_1,x_2,x_3$ in $U$ such that
$$\forall i \in \{1,2,3\}, \; \forall s \in \calS, \quad F(s)=s(x_i) \mod \K y_i.$$
As $F$ is non-local, the vectors $x_1,x_2,x_3$ must be pairwise distinct. Replacing $F$ with $s \mapsto F(s)-s(x_3)$,
we lose no generality in assuming that $x_3=0$.
Then, $x_1$ and $x_2$ are distinct non-zero vectors of $U$ and, for all $s \in \calS$,
\begin{equation}\label{belongid}
s(x_1) \in \Vect(y_1,y_3), \quad  s(x_2) \in \Vect(y_2,y_3) \quad \text{and} \quad s(x_1-x_2) \in \Vect(y_1,y_2).
\end{equation}
As $\Vect(y_1,y_3) \cap \Vect(y_2,y_3) \cap \Vect(y_1,y_3)=\{0\}$, the collinearity of $x_1$ with $x_2$ would entail that
$s(x_1)=0$ for all $s \in \calS$, leading to $\codim \calS \geq \dim V$ and thereby contradicting our assumptions.
Therefore, $x_1$ and $x_2$ are non-collinear.
Denote then by $\calT$ the linear subspace of $\calL(U,V)$ consisting of all the operators $s$ that satisfy \eqref{belongid}.
Clearly, $\codim \calT>2\dim V-4$. On the other hand $\calS \subset \calT$ whence $\codim \calT \leq \codim \calS = \dim V-1$.
This would lead to $\dim V<3$, in contradiction with an earlier assumption.

Thus, there is a $2$-dimensional linear subspace $P$ of $V$ such that $\dim S^\bot y \leq 1$
for all $y \in V \setminus (P \cup D_0 \cup D_1)$.
From there, we deduce that $\dim S^\bot y \leq 1$ for all $y \in V$: indeed,
with the same line of reasoning as in the proof of Lemma \ref{quasiexistquadform},
the contrary would yield a non-zero quadratic form $q$ on $V$
that vanishes everywhere on $V \setminus (P \cup D_0 \cup D_1)$, whereas Lemma \ref{quasipropquadform} shows that no such quadratic form exists.
Combining $\codim S=\dim V-1$ with Lemma \ref{quasialternative},
we deduce that there exists a vector $x \in U$ such that $\dim Sx=1$, and hence $\dim \calS x=1$.

This completes the proof of Theorem \ref{quasiaffinetheo2}.

\subsection{Proof of Theorem \ref{quasiaffinetheo3}}

Denote by $S$ the translation vector space of $\calS$.

First, the assumptions show that $\calS$ is exactly the affine subspace consisting of all the operators $s \in \calL(U,V)$ such that $s(x) \in \calS x$
(because it is included in it and both of them have codimension $\dim V-1$ in $\calL(U,V)$).
In particular all the non-zero operators in $S^\bot$ have range $\K x$.

It follows that the intersection $D_1$ of the kernels of the rank $1$ operators in $S^\bot$ has dimension $1$.
Note that $D_1$ is the translation vector space of $\calS x$!
By Lemma \ref{secondlinelemma}, we obtain that $\vec{F}(s) \in \im s$ for every rank $1$ operator $s \in S$ such that $\im s$ differs from both $D_0$ and $D_1$.
Next, set
$$P:=\Vect(\calS x)$$
and note that $P$ is a $2$-dimensional linear subspace of $V$.

Assume for the moment that
\begin{equation}\label{nullonsx}
\forall s \in S, \quad s(x)=0 \Rightarrow \vec{F}(s)=0.
\end{equation}
This yields an affine map $\psi : \calS x \rightarrow V$ such that
$$\forall s \in \calS, \quad F(s)=\psi(s(x)).$$
Note that $\# (\calS x \setminus D_0) \geq 2$.
For all $y \in (\calS x) \setminus D_0$, we can find an operator $s \in \calS$ such that $\im s=\K y$ and $s(x)=y$, leading to
$\psi(y) \in \K y$ since $y \not\in (D_0 \cup D_1)$. Denoting by $\widetilde{\psi} : P \rightarrow V$
the linear extension of $\psi$, we deduce that $\widetilde{\psi}$ is an endomorphism of $P$
and that every $y \in (\calS x) \setminus D_0$ is an eigenvector of it. In particular, we have just proved statements (iii) and (iv)
(still assuming that \eqref{nullonsx} holds).
In that case, if in addition $\# \K>3$ or $D_0 \cap (\calS x)=\emptyset$, we deduce that $\widetilde{\psi} : z \mapsto \lambda z$
for some fixed scalar $\lambda$, and hence $F : s \mapsto s(\lambda x)$, contradicting the assumption that $F$ should be non-local.
Thus, if condition \eqref{nullonsx} holds then the conclusion of Theorem \ref{quasiaffinetheo3} is satisfied.

\vskip 3mm
In the rest of the proof we try to reduce the situation to the one where condition \eqref{nullonsx} holds.
Note that neither the assumptions nor the conclusions of Theorem \ref{quasiaffinetheo3} are modified in subtracting a local map from $F$.

Let us choose distinct vectors $y_1$ and $y_2$ in $\calS x \setminus D_0$.
Note then that $y_1$ and $y_2$ are non-collinear. Moreover, for all $i \in \{1,2\}$,
$\calS \modu y_i$ is actually a linear subspace of $\calL(U,V/\K y_i)$, and
$(\calS \modu y_i) x =P/ \K y_i$ has dimension 1.
For all $i \in \{1,2\}$, we have a well-defined quasi-range-compatible affine map
$F \modu y_i  : \calS \modu y_i \rightarrow V/\K y_i$. As $\calS \modu y_i$ is a linear subspace of
$\calL(U,V/\K y_i)$, the map $F \modu y_i$ is actually linear.

From there, we split the discussion into two cases.

\vskip 3mm
\noindent \textbf{Case 1: $\dim V>2$.} \\
Let $i \in \{1,2\}$.
By Proposition \ref{quasidegenerate1} applied to $F \modu y_i$, we can write
$$(F \modu y_i): \; s \mapsto \varphi_i(s(x))+s(x_i)$$
for some linear map $\varphi_i : (\calS \modu y_i) x \rightarrow V/(\K y_i)$
and some vector $x_i \in (U \setminus \K x) \cup \{0\}$.
Replacing $F$ with $s \mapsto F(s)-s(x_1)$, we see that no generality is lost in assuming that $x_1=0$
(we do not replace $x_2$ with $x_2-x_1$, rather we apply Proposition \ref{quasidegenerate1} to the new map
$F \modu y_2$).
Note that the linear part $\vec{F}$ of $F$ then satisfies
\begin{equation}\label{linearpartequation}
\forall s \in S, \quad s(x)=0 \Rightarrow \bigl(\vec{F}(s) \in \K y_1 \quad \text{and} \quad \vec{F}(s)=s(x_2) \mod \K y_2  \bigr).
\end{equation}

If $x_2 \neq 0$, then $x$ and $x_2$ are linearly independent; as $\dim V>2$ there would then exist
$s \in S$ such that $s(x)=0$ and $s(x_2) \not\in P$, contradicting condition \eqref{linearpartequation}.
Therefore, $x_2=0$. Using \eqref{linearpartequation}, we see that condition \eqref{nullonsx} holds.

\vskip 3mm
\noindent \textbf{Case 2: $\dim V=2$.} \\
Then, $P=V$ and hence conclusion (iii) is trivially true.
Let $i \in \{1,2\}$. Then, as $\calS \modu y_i$ is the full space $\calL(U,V/\K y_i)$ and as $\dim (V/\K y_i)=1$, every linear map
from $\calL(U,V/\K y_i)$ to $V/\K y_i$ is local. This yields
vectors $x_1,x_2$ in $U$ such that
$$\forall i \in \{1,2\}, \; \forall s \in \calS, \quad F(s)=s(x_i) \mod \K y_i.$$
Subtracting the local map $s \mapsto s(x_1)$ from $F$, we lose no generality in assuming that $x_1=0$.

If $x_2 \in \K x$, then condition \eqref{nullonsx} holds and the conclusion follows.

Assume finally that $x_2 \not\in \K x$.
If $\# \K>3$ or $D_0 \cap \calS x=\emptyset$, then we can choose $y_3 \in \calS x$ that belongs to none of $\K y_1$, $\K y_2$ and $D_0$;
then, we can choose $s \in \calS$ such that $s(x_2)=y_3$ and $\im s=\K y_3$;
it follows that $F(s) \in \im s \cap \K y_1=\{0\}$; hence, $s(x_2) \in \K y_2$, contradicting the assumption that
$y_3 \not\in \K y_2$. Therefore, $\# \K=3$ and $D_0 \cap \calS x \neq \emptyset$. Thus, conclusions (i), (ii) and (iii) are satisfied,
which completes the proof in the case when $\dim V=2$.

Thus, our last theorem is established.


\begin{thebibliography}{1}
\bibitem{AtkLloyd}
M. D. Atkinson, S. Lloyd,
{Large spaces of matrices of bounded rank.}
Quart. J. Math. Oxford (2)
{\bf 31} (1980), 253--262.

\bibitem{Dieudonne}
J. Dieudonn\'e,
{Sur une g\'en\'eralisation du groupe orthogonal \`a quatre variables.}
Arch. Math.
{\bf 1} (1948), 282--287.

\bibitem{dSPRC1}
C. de Seguins Pazzis,
{Range-compatible homomorphisms on matrix spaces.}
Preprint, arXiv: http://arxiv.org/abs/1307.3574

\bibitem{dSPRC2}
C. de Seguins Pazzis,
{Range-compatible homomorphisms over the field with two elements.}
Preprint, arXiv: http://arxiv.org/abs/1407.4077

\bibitem{dSPfeweigenvalues}
C. de Seguins Pazzis,
{Spaces of matrices with few eigenvalues.}
Linear Algebra Appl.
{\bf 449} (2014), 210--311.

\bibitem{dSPclass}
C. de Seguins Pazzis,
{The classification of large spaces of matrices with bounded rank.}
Israel J. Math.
(2015), in press.

\bibitem{dSPlargelinpres}
C. de Seguins Pazzis,
{The linear preservers of non-singularity in a large space of matrices.}
Linear Algebra Appl.
{\bf 436} (2012), 3507--3530.

\end{thebibliography}
\end{document}